\documentclass[11pt]{article}

\usepackage{amsxtra,amssymb,amsthm,amsmath,latexsym}
\usepackage{graphicx}
\usepackage{epsfig}
\usepackage{afterpage}
\newtheorem{thm}{Theorem}[section]

\newtheorem{lem}[thm]{Lemma}
 \newcommand{\thmref}[1]{Theorem~\ref{#1}}
 \newcommand{\lemref}[1]{Lemma~\ref{#1}}

\newcommand{\R}{{\mathbb R}}

\newcommand{\la}{{\langle}}
\newcommand{\ra}{{\rangle}}
\newcommand{\dl}{{\delta}}

\newcommand{\bee}{\begin{equation*}}
\newcommand{\eee}{\end{equation*}}
\newcommand{\be}{\begin{equation}}
\newcommand{\ee}{\end{equation}}
\newcommand{\pn}{\par\noindent}
\title{Inversion of the Laplace transform from the real axis using an
adaptive iterative method}
\author{Sapto W. Indratno\\
\small Department of Mathematics\\[-0.8ex]
\small Kansas State University, Manhattan, KS 66506-2602, USA\\
\small \texttt{sapto@math.ksu.edu}%\\
\and
A G Ramm\\
\small Department of Mathematics\\[-0.8ex]
\small Kansas State University, Manhattan, KS 66506-2602, USA\\[-0.8ex]
\small \texttt{ramm@math.ksu.edu}\\
}

\begin{document}
\date{}
% typeset front matter
\maketitle

\begin{abstract}
In this paper a new method for inverting the Laplace transform from
the real axis is formulated. This method is based on a quadrature
formula. We assume that the unknown function $f(t)$ is continuous
with (known) compact support. An adaptive iterative method and an
adaptive stopping rule, which yield the convergence of the
approximate solution to $f(t)$, are proposed in this paper.
\end{abstract}
\pn{
{\em Key words:} Fredholm integral equations of the first kind;
(adaptive)iterative regularization; inversion of the Laplace
transform; discrepancy principle }
\section{Introduction}
Consider the Laplace transform : \be\label{Lf}
\mathcal{L}f(p):=\int_0^\infty e^{-pt}f(t)dt=F(p),\quad
\text{Re}p>0, \ee where $\mathcal{L}:X_{0,b}\to L^2[0,\infty)$,
\be\label{Xb} X_{0,b}:=\{f\in L^2[0,\infty)\ | \ \text{supp}
f\subset [0,b)\},\quad b>0.\ee We assume in \eqref{Xb} that $f$ has
compact support. This is not a restriction practically. Indeed, if
$\lim_{t\to \infty}f(t)=0$, then $|f(t)|<\dl$ for $t>t_\dl$, where
$\dl>0$ is an arbitrary small number. Therefore, one may assume that
supp$f\subset [0,t_\delta]$, and treat the values of $f$ for
$t>t_\delta$ as noise. One may also note that if $f\in
L^1(0,\infty)$, then
$$F(p):=\int_0^\infty f(t)e^{-pt}dt=\int_0^bf(t)e^{-pt}dt+\int_b^\infty f(t)e^{-pt}dt:=F_1(p)
+F_2(p),$$ and $|F_2(p)|\leq e^{-bp}\delta$, where
$\int_b^\infty|f(t)|dt\leq \delta$. Therefore, the contribution of
the "tail" $f_b(t)$ of $f$, $$f_b(t):=\left\{
                                       \begin{array}{ll}
                                         0, & \hbox{$t<b$,} \\
                                         f(t), & \hbox{$t\geq b$,}
                                       \end{array}
                                     \right.$$ can be considered as noise if $b>0$ is large and
$\delta>0$ is small. We assume in \eqref{Xb} that $f\in
L^2[0,\infty)$. One may also assume that $f\in L^1[0,\infty)$, or
that $|f(t)|\leq c_1e^{c_2t}$, where $c_1, c_2 $ are positive
constants. If the last assumption holds, then one may define the
function $g(t):= f(t)e^{-(c_2+1)t}$. Then $g(t)\in L^1[0,\infty)$,
and its Laplace transform $G(p)=F(p+c_2+1)$ is known on the interval
$[c_2+1, c_2+1+b]$ of real axis if the Laplace transform $F(p)$ of
$f(t)$ is known on the interval $[0,b]$. Therefore, our inversion
methods are applicable to these more general classes of functions
$f$ as well.

The operator $\mathcal{L}:X_{0,b}\to L^2[0,\infty)$ is compact.
Therefore, the inversion of the Laplace transform \eqref{Lf} is an
ill-posed problem (see \cite{MRZ84}, \cite{RAMM05}). Since the
problem is ill-posed, a regularization method is needed to obtain a
stable inversion of the Laplace transform. There are many methods to
solve equation \eqref{Lf} stably: variational regularization,
quasisolutions, iterative regularization (see e.g, \cite{SWIAGR09},
\cite{MRZ84}, \cite{RAMM05}, \cite{RAMM499}). In this paper we
propose an adaptive iterative method based on the Dynamical Systems
Method (DSM) developed in \cite{RAMM05}, \cite{RAMM499}. Some
methods have been developed earlier for the inversion of the Laplace
transform (see \cite{KSC76}, \cite{BDBM79}, \cite{HDJA68},
\cite{PI06}). In many papers the data $F(p)$ are assumed exact and
given on the complex axis. In \cite{VVK06} it is shown that the
results of the inversion of the Laplace transform from the complex
axis are more accurate than these of the inversion of the Laplace
transform from the real axis. The reason is the ill-posedness of the
Laplace transform inversion from the real axis. A survey regarding
the methods of the Laplace transform inversion has been given in
\cite{BDBM79}. There are several types of the Laplace inversion
method compared in \cite{BDBM79}. The inversion formula for the
Laplace transform is well known: \be\label{Mellin}
f(t)=\frac{1}{2\pi
i}\int_{\sigma-i\infty}^{\sigma+i\infty}F(p)e^{pt}dp,\quad \sigma>0,
\ee is used in some of these methods, and then $f(t)$ is computed by
some quadrature formulas, and many of these formulas can be found in
\cite{PDPR84} and \cite{VKNS68}. Moreover, the ill-posedness of the
Laplace transform inversion is not discussed in all the methods
compared in \cite{BDBM79}. The approximate $f(t)$, obtained by these
methods when the data are noisy, may differ significantly from
$f(t)$. There are some papers in which the inversion of the Laplace
transform from the real axis was studied (see \cite{AYAGR00},
\cite{LDAM02}, \cite{CWD03}, \cite{FMSS}, \cite{VVK06},
\cite{MCDG07}, \cite{RAMM198}, \cite{VAR83}, \cite{JWEP78}). In
\cite{AYAGR00} and \cite{RAMM198} a method based on the Mellin
transform is developed. In this method the Mellin transform of the
data $F(p)$ is calculated first and then inverted for $f(t)$. In
\cite{LDAM02} a Fourier series method for the inversion of Laplace
transform from the real axis is developed. The drawback of this
method comes from the ill-conditioning of the discretized problem.
It is shown in \cite{LDAM02} that if one uses some basis functions
in $X_{0,b}$, the problem becomes extremely ill-conditioned if the
number $m$ of the basis functions exceeds $20$. In \cite{FMSS} a
reproducing kernel method is used for the inversion of the Laplace
transform. In the numerical experiments in \cite{FMSS} the authors
use double and multiple precision methods to obtain high accuracy
inversion of the Laplace transform. The usage of the multiple
precision increases the computation time significantly which is
observed in \cite{FMSS}, so this method may be not efficient in
practice. A detailed description of the multiple precision technique
can be found in \cite{HF} and \cite{KMH67}. Moreover, the Laplace
transform inversion with perturbed data is not discussed in
\cite{FMSS}. In \cite{JWEP78} the authors develop an inversion
formula, based on the eigenfunction expansion for the Laplace
transform. The difficulties with this method are: a) the inversion
formula is not applicable when the data are noisy, b) even for exact
data the inversion formula is not suitable for numerical
implementation.

The Laplace transform as an operator from $C_{0k}$ into $L^2$, where
$C_{0k}=\{f(t)\in C[0,+\infty)\ | \ \text{supp} f\subset [0,k)\},\
k=const>0, \ L^2:=L^2[0,\infty),$ is considered in \cite{CWD03}. The
finite difference method is used in \cite{CWD03} to discretize the
problem, where the size of the linear algebraic system obtained by
this method is fixed at each iteration, so the computation time
increases if one uses large linear algebraic systems. The method of
choosing the size of the linear algebraic system is not given in
\cite{CWD03}. Moreover, the inversion of the Laplace transform when
the data $F(p)$ is given only on a finite interval $[0,d]$, $d>0$,
is not discussed in \cite{CWD03}.

The novel points in our paper are:
\begin{itemize}
\item[1)] the representation of the approximation solution
\eqref{fmdel} of the function $f(t)$ which depends only on the
kernel of the Laplace transform,
\item[2)]the adaptive iterative
scheme \eqref{it3} and adaptive stopping rule \eqref{dp}, which
generate the regularization parameter, the discrete data
$F_\dl(p)$ and the number of terms in \eqref{fmdel}, needed for
obtaining an approximation of the unknown function $f(t)$.
\end{itemize}
We study the inversion problem using the pair of spaces
$(X_{0,b},L^2[0,d])$, where $X_{0,b}$ is defined in \eqref{Xb},
develop an inversion method, which can be easily implemented
numerically, and demonstrate in the numerical experiments that our
method yields the results comparable in accuracy with the results,
presented in the literature, e.g., with the double precision results
given in paper \cite{FMSS}.

The smoothness of the kernel allows one to use the compound
Simpson's rule in approximating the Laplace transform. Our approach
yields a representation \eqref{fmdel} of the approximate inversion
of the Laplace transform. The number of terms in approximation
\eqref{fmdel} and the regularization parameter are generated
automatically by the proposed adaptive iterative method. Our
iterative method is based on the iterative method proposed in
\cite{SWIAGR092}. The adaptive stopping rule we propose here is
based on the discrepancy-type principle, established in
\cite{RAMM525}. This stopping rule yields convergence of the
approximation \eqref{fmdel} to $f(t)$ when the noise level $\dl\to
0$.

A detailed derivation of our inversion method is given in Section 2.
In Section 3 some results of the numerical experiments are reported.
These results demonstrate the efficiency and stability of the
proposed method.

\section{Description of the method}
Let $f\in X_{0,b}$. Then equation \eqref{Lf} can be written as: \be
(\mathcal{L}f)(p):=\int_0^be^{-pt}f(t)dt=F(p),\quad 0\leq p. \ee Let
us assume that the data $F(p)$, the Laplace transform of $f$, are
known only for $0\leq p\leq d<\infty.$ Consider the mapping
$\mathcal{L}_m: L^2[0,b]\to \R^{m+1}$, where \be\label{amkb}
(\mathcal{L}_m f)_i:=\int_0^be^{-p_it}f(t)dt=F(p_i),\quad
i=0,1,2,\hdots,m, \ee \be\label{pj}p_i:=ih,\quad i=0,1,2,\hdots,m,\
h:=\frac{d}{m},\ee and $m$ is an even number which will be chosen
later. Then the unknown function $f(t)$ can be obtained from a
finite-dimensional operator equation \eqref{amkb}. Let
\be\label{ipWm} \la u,v\ra_{W^{m}}:=\sum_{j=0}^{m}w_j^{(m)}
u_jv_j\quad and \quad \|u\|_{W^{m}}:=\la u,u\ra_{W^m} \ee be the
inner product and norm in $\R^{m+1}$, respectively, where
$w_j^{(m)}$ are the weights of the compound Simpson's rule (see
\cite[p.58]{PDPR84}), i.e., \be\label{wj}w_j^{(m)}:=\left\{
                          \begin{array}{ll}
                            h/3, & \hbox{$j=0,m$;} \\
                            4h/3, & \hbox{$j=2l-1,\ l=1,2,\hdots,m/2$;} \\
                            2h/3, & \hbox{$j=2l,\ l=1,2,\hdots,(m-2)/2$,}
                          \end{array}
                        \right.\quad h=\frac{d}{m},
\ee where $m$ is an even number. Then \be\begin{split} \la
\mathcal{L}_mg,v\ra_{W^m}&=\sum_{j=0}^m
w_j^{(m)}\int_0^be^{-p_jt}g(t)dt v_j\\
&=\int_0^b g(t)\sum_{j=0}^m w_j^{(m)}e^{-p_jt}v_j dt=\la
g,\mathcal{L}_m^*v\ra_{X_{0,b}},
\end{split}\ee where \be\label{amkbs} \mathcal{L}_m^*v=\sum_{j=0}^mw_j^{(m)}
e^{-p_jt}v_j,\quad v:=\left(
                                \begin{array}{c}
                                  v_0 \\
                                  v_1 \\
                                  \vdots \\
                                  v_m \\
                                \end{array}
                              \right)
\in \R^{m+1}. \ee and \be\label{ipX}\la
g,h\ra_{X_{0,b}}:=\int_0^bg(t)h(t)dt.\ee  It follows from
\eqref{amkb} and \eqref{amkbs} that \be\label{Tm}
(\mathcal{L}_m^*\mathcal{L}_mg)(t)=\sum_{j=0}^m
w_j^{(m)}e^{-p_jt}\int_0^b e^{-p_j z}g(z)dz :=(T^{(m)}g)(t),\ee and
\be\label{Qm} \mathcal{L}_m\mathcal{L}_m^*v=\left(
                                                  \begin{array}{c}
                                                    \int_0^b e^{-p_0 t}\sum_{j=0}^mw_j^{(m)}
e^{-p_jt}v_j dt\\
                                                    \int_0^b e^{-p_1 t}\sum_{j=0}^mw_j^{(m)}
e^{-p_jt}v_j dt\\
                                                    \vdots \\
                                                    \int_0^b e^{-p_mt}\sum_{j=0}^mw_j^{(m)}
e^{-p_jt}v_j dt\\
                                                  \end{array}
                                                \right):=Q^{(m)}v,
 \ee
where \be\label{Bm} (Q^{(m)})_{ij}:=w_j^{(m)}\int_0^b
e^{-(p_i+p_j)t}dt=w_j^{(m)}\frac{1-e^{-b(p_i+p_j)}}{p_i+p_j},\quad
i,j=0,1,2,\hdots,m.\ee
\begin{lem}\label{lemswj}
Let $w_j^{(m)}$ be defined in \eqref{wj}. Then \be\label{swj}
\sum_{j=0}^mw_j^{(m)}=d, \ee for any even number $m$.
\end{lem}
\begin{proof}
From definition \eqref{wj} one gets \be\begin{split}
\sum_{j=0}^mw_j^{(m)}&=w_0^{(m)}+w_m^{(m)}+\sum_{j=1}^{m/2}w_{2j-1}^{(m)}+\sum_{j=1}^{(m-2)/2}w_{2j}^{(m)}\\
&=\frac{2h}{3}+\sum_{j=1}^{m/2}\frac{4h}{3}+\sum_{j=1}^{(m-2)/2}\frac{2h}{3}\\
&=\frac{2h}{3}+\frac{2hm}{3}+\frac{h(m-2)}{3}=hm=\frac{d}{m}m=d.
\end{split}\ee
\lemref{lemswj} is proved.
\end{proof}
\begin{lem}
The matrix $Q^{(m)}$, defined in \eqref{Bm}, is positive
semidefinite and self-adjoint in $\R^{m+1}$ with respect to the
inner product \eqref{ipWm}.
\end{lem}
\begin{proof}
Let \be (H_m)_{ij}:=\int_0^b
e^{-(p_i+p_j)t}dt=\frac{1-e^{-b(p_i+p_j)}}{p_i+p_j}, \ee and \be
(D_m)_{ij}=\left\{
                 \begin{array}{ll}
                   w_i^{(m)}, & \hbox{$i=j$;} \\
                   0, & \hbox{otherwise,}
                 \end{array}
               \right.
 \ee $w_j^{(m)}$ are defined in \eqref{wj}.
Then $\la D_mH_mD_m u,v\ra_{\R^{m+1}}=\la
u,D_mH_mD_mv\ra_{\R^{m+1}},$ where \be\label{ipr}\la
u,v\ra_{\R^{m+1}}:=\sum_{j=0}^{m}u_jv_j,\quad u,v\in\R^{m+1}.\ee We
have \be\begin{split} \la
Q^{(m)}u,v\ra_{W^{m}}&=\sum_{j=0}^mw_j^{(m)} (Q^{(m)}
u)_jv_j=\sum_{j=0}^{m} (D_mH_mD_m u)_jv_j\\
&=\la D_mH_mD_mu,v\ra_{\R^{m+1}}=\la u,D_mH_mD_mv\ra_{\R^{m+1}}\\
&=\sum_{j=0}^{m}u_j(D_mH_mD_mv)_j=\sum_{j=0}^{m}u_jw_j^{(m)}(H_mD_mv)_j\\
&=\la u,Q^{(m)}v\ra_{W^{m}}.
\end{split}\ee Thus, $Q^{(m)}$ is self-adjoint with respect to inner
product \eqref{ipWm}. We have \be\begin{split} (H_m)_{ij}&=\int_0^b
e^{-(p_i+p_j)t}dt=\int_0^b
e^{-p_it}e^{-p_jt}dt \\
&=\la \phi_i,\phi_j\ra_{X_{0,b}},\quad
\phi_i(t):=e^{-p_it},\end{split}\ee where $\la
\cdot,\cdot\ra_{X_{0,b}}$ is defined in \eqref{ipX}. This shows that
$H_m$ is a Gram matrix. Therefore, \be\label{Huu} \la
H_mu,u\ra_{\R^{m+1}}\geq 0,\ \forall u\in \R^{m+1}. \ee This implies
\be \la Q^{(m)}u,u\ra_{W^{m}}=\la Q^{(m)}u, D_mu\ra_{\R^{m+1}}=\la
H_mD_mu,D_mu\ra_{\R^{m+1}}\geq 0. \ee Thus, $Q^{(m)}$ is a positive
semidefinite and self-adjoint matrix with respect to the inner
product \eqref{ipWm}.
\end{proof}

\begin{lem}\label{lemTm}
Let $T^{(m)}$ be defined in \eqref{Tm}. Then $T^{(m)}$ is
self-adjoint and positive semidefinite operator in $X_{0,b}$ with
respect to inner product \eqref{ipX}.
\end{lem}
\begin{proof}
From definition \eqref{Tm} and inner product \eqref{ipX} we get
\be\begin{split} \la T^{(m)}g,h\ra_{X_{0,b}}&=\int_0^b\sum_{j=0}^m
w_j^{(m)}e^{-p_jt}\int_0^b e^{-p_j z}g(z)dzh(t)dt \\
&=\int_0^bg(z)\sum_{j=0}^mw_j^{(m)}e^{-p_jz}\int_0^be^{-p_jt}h(t)dtdz\\
&=\la g,T^{(m)}h\ra_{X_{0,b}}.
\end{split}\ee Thus, $T^{(m)}$ is a self-adjoint operator with respect to inner product
\eqref{ipX}. Let us prove that $T^{(m)}$ is positive semidefinite.
Using \eqref{Tm}, \eqref{wj}, \eqref{ipWm} and \eqref{ipX}, one gets
\be
\begin{split}\la T^{(m)}g,g\ra_{X_{0,b}}&=\int_0^b\sum_{j=0}^m
w_j^{(m)}e^{-p_jt}\int_0^b e^{-p_j z}g(z)dzg(t)dt \\
&=\sum_{j=0}^mw_j^{(m)}\int_0^b e^{-p_j z}g(z)dz\int_0^b e^{-p_j
t}g(t)dt\\
&=\sum_{j=0}^mw_j^{(m)}\left(\int_0^b e^{-p_j z}g(z)dz\right)^2\geq
0.
\end{split}\ee
\lemref{lemTm} is proved.
\end{proof}
From \eqref{amkbs} we get
Range$[\mathcal{L}_m^*]=span\{w_j^{(m)}k(p_j,\cdot,0)\}_{j=0}^m,$
where \be\label{kptz}k(p,t,z):=e^{-p(t+z)}.\ee Let us approximate
 the unknown $f(t)$  as follows:
\be\label{aps}f(t)\approx\sum_{j=0}^m c_j^{(m)}
w_j^{(m)}e^{-p_jt}=T_{a,m}^{-1}\mathcal{L}_m^*F^{(m)}:=f_m(t),\ee
where $p_j$ are defined in \eqref{pj}, $T_{a,m}$ is defined in
\eqref{QTam}, and $c_j^{(m)}$ are constants obtained by solving the
linear algebraic system: \be\label{LAS0}
(aI+Q^{(m)})c^{(m)}=F^{(m)}, \ee where $Q^{(m)}$ is defined in
\eqref{Qm}, \be c^{(m)}:=\left(
                                                \begin{array}{c}
                                                  c_0^{(m)} \\
                                                  c_1^{(m)} \\
                                                  \vdots \\
                                                  c_m^{(m)} \\
                                                \end{array}
                                              \right)\quad and\quad F^{(m)}:=\left(
                                                \begin{array}{c}
                                                  F(p_0) \\
                                                  F(p_1) \\
                                                  \vdots \\
                                                  F(p_m) \\
                                                \end{array}
                                              \right).
 \ee
To prove the convergence of the approximate solution $f(t)$, we use
the following estimates, which are proved in \cite{RAMM499}, so
their proofs are omitted.
\begin{lem}\label{aBm} Let $T^{(m)}$ and $Q^{(m)}$ be defined in
\eqref{Tm} and \eqref{Qm}, respectively. Then, for $a>0$, the
following estimates hold:
\be\label{eQam}\|Q_{a,m}^{-1}\mathcal{L}_m\|\leq
\frac{1}{2\sqrt{a}},\ee \be\label{eaQam}a\|Q_{a,m}^{-1}\|\leq 1,\ee
\be\label{eTam}\|T_{a,m}^{-1}\|\leq \frac{1}{a},\ee
\be\label{eTLm}\|T_{a,m}^{-1}\mathcal{L}_m^*\|\leq
\frac{1}{2\sqrt{a}},\ee where \be\label{QTam} Q_{a,m}:=Q^{(m)}+aI
\quad T_{a,m}:=T^{(m)}+aI, \ee $I$ is the identity operator and
$a=const>0.$
\end{lem}
Estimates \eqref{eQam} and \eqref{eaQam} are used in proving
inequality \eqref{acm}, while estimates \eqref{eTam} and
\eqref{eTLm} are used in the proof of lemmas 2.9 and 2.10,
respectively.

Let us formulate an iterative method for obtaining the approximation
solution of $f(t)$ with the exact data $F(p)$. Consider the
following iterative scheme \be\label{it1}
u_n(t)=qu_{n-1}(t)+(1-q)T_{a_n}^{-1}\mathcal{L}^*F,\quad u_0(t)=0,
\ee where $\mathcal{L}^*$ is the adjoint of the operator
$\mathcal{L}$, i.e., \be\label{Lad}
(\mathcal{L}^*g)(t)=\int_0^de^{-pt}g(p)dp, \ee
\be\label{T}\begin{split}
(Tf)(t)&:=(\mathcal{L}^*\mathcal{L}f)(t)=\int_0^b \int_0^dk(p,t,z)dp
f(z)dz\\
&=\int_0^b\frac{f(z)}{t+z}\left(1-e^{-d(t+z)}\right)dz,\end{split}\ee
$k(p,t,z)$ is defined in \eqref{kptz}, \be\label{Ta} T_a:=aI+T,\quad
a>0, \ee \be\label{san} a_n:=qa_{n-1},\quad a_0>0,\quad q\in(0,1).
\ee
\begin{lem}\label{lem12}
Let $T_a$ be defined in \eqref{Ta}, $\mathcal{L}f=F$, and $f\perp
\mathcal{N}(\mathcal{L})$, where $\mathcal{N}(\mathcal{L})$ is the
null space of $\mathcal{L}$. Then \be a\|T_a^{-1}f\|\to 0\quad
\text{as } a\to 0. \ee
\end{lem}
\begin{proof}
Since $f\perp \mathcal{N}(\mathcal{L})$, it follows from the
spectral theorem that
$$\lim_{a\to 0}a^2\|T_a^{-1}f\|^2=\lim_{a\to 0}
\int_0^\infty\frac{a^2}{(a+s)^2}d\langle
E_sf,f\rangle=\|P_{\mathcal{N}(\mathcal{L})}f\|^2=0,$$ where $E_s$
is the resolution of the identity corresponding to
$\mathcal{L}^*\mathcal{L}$, and $P$ is the
orthogonal projector onto $\mathcal{N}(\mathcal{L})$.\\
\lemref{lem12} is proved.
\end{proof}
\begin{thm}\label{thm1}
Let $\mathcal{L}f=F$, and $u_n$ be defined in \eqref{it1} Then
\be\label{rel1} \lim_{n\to \infty}\|f-u_n\|=0. \ee
\end{thm}
\begin{proof}
By induction we get \be u_n=\sum_{j=0}^{n-1}\omega_j^{(n)}
T_{a_{j+1}}^{-1}\mathcal{L}^*F, \ee where $T_a$ is defined in
\eqref{Ta}, and\be\label{omg} \omega_j^{(n)}:=q^{n-j-1}-q^{n-j}. \ee
Using the identities \be \mathcal{L}f=F, \ee \be
T_a^{-1}\mathcal{L}^*\mathcal{L}=T_a^{-1}(T+aI-aI)=I-aT_{a}^{-1}\ee
and \be \sum_{j=0}^{n-1}\omega_j^{(n)}=1-q^{n}, \ee
 we get \be\begin{split}
f-u_n&=f-\sum_{j=0}^{n-1}\omega_j^{(n)}f+\sum_{j=0}^{n-1}\omega_j^{(n)}a_{j+1}T_{a_{j+1}}^{-1}f\\
&=q^nf+\sum_{j=0}^{n-1}\omega_j^{(n)}a_{j+1}T_{a_{j+1}}^{-1}f.\end{split}
\ee Therefore, \be\label{es1} \|f-u_n\|\leq q^n\|f\|+
\sum_{j=0}^{n-1}\omega_j^{(n)}a_{j+1}\|T_{a_{j+1}}^{-1}f\|.\ee To
prove relation \eqref{rel1} the following lemma is needed:
\begin{lem}\label{lemq} Let $g(x)$ be a continuous function on
$(0,\infty)$, $c>0$ and $q\in(0,1)$ be constants. If \be\label{g}
\lim_{x\to 0^+}g(x)=g(0):=g_0,\ee then \be\label{rel8} \lim_{n\to
\infty}\sum_{j=0}^{n-1}\left(q^{n-j-1}-q^{n-j}\right)g(cq^{j+1})=
g_0. \ee
\end{lem}
\begin{proof}
Let
\be\label{Fn}F_l(n):=\sum_{j=1}^{l-1}\omega_j^{(n)}g(cq^{j+1}),\ee
where $\omega_j^{(n)}$ are defined in \eqref{omg}. Then \bee
|F_{n+1}(n)-g_0|\leq |F_{l}(n)|+\left|\sum_{j=l}^n
\omega_j^{(n)}g(cq^{j+1})-g_0\right|.\eee Take $\epsilon>0$
arbitrarily small. For sufficiently large fixed $l(\epsilon)$ one
can choose $n(\epsilon)>l(\epsilon)$, such that \bee
|F_{l(\epsilon)}(n)|\leq \frac{\epsilon}{2},\ \forall n>n(\epsilon),
\eee because $\lim_{n\to\infty}q^n=0.$ Fix $l=l(\epsilon)$ such that
$|g(cq^j)-g_0|\leq \frac{\epsilon}{2}$ for $j>l(\epsilon)$. This is
possible because of \eqref{g}. One has \bee |F_{l(\epsilon)}(n)|\leq
\frac{\epsilon}{2},\ n>n(\epsilon)>l(\epsilon) \eee and
\bee\begin{split} \left|\sum_{j=l(\epsilon)}^n
\omega_j^{(n)}g(cq^{j+1})-g_0\right|&\leq \sum_{j=l(\epsilon)}^{n}
\omega_j^{(n)}|g(cq^{j+1})-g_0|+|\sum_{j=l(\epsilon)}^{n}
\omega_j^{(n)}-1||g_0|\\
&\leq
\frac{\epsilon}{2}\sum_{j=l(\epsilon)}^n\omega_j^{(n)}+q^{n-l(\epsilon)}|g_0|\\
&\leq \frac{\epsilon}{2}+|g_0|q^{n-l(\epsilon)}\leq
\epsilon,\end{split}\eee if $n(\epsilon)$ is sufficiently large.
Here we have used the relation\bee
\sum_{j=l}^{n}\omega_j^{(n)}=1-q^{n-l}. \eee Since $\epsilon>0$ is
arbitrarily small, relation \eqref{rel8} follows.\\\lemref{lemq} is
proved.
\end{proof}

\lemref{lem12} together with \lemref{lemq} with
$g(a)=a\|T_a^{-1}f\|$ yield \be \lim_{n\to
\infty}\sum_{j=0}^{n-1}\omega_j^{(n)}a_{j+1}\|T_{a_{j+1}}^{-1}f\|=0.
\ee This together with estimate \eqref{es1} and condition
$q\in(0,1)$ yield relation \eqref{rel1}.\\
\thmref{thm1} is proved.
\end{proof}
\begin{lem}\label{lem1}
Let $T$ and $T^{(m)}$ be defined in \eqref{T} and \eqref{Tm},
respectively. Then \be\label{TTm} \|T-T^{(m)}\|\leq
\frac{(2bd)^5}{540\sqrt{10} m^4}. \ee
\end{lem}
\begin{proof}
From definitions \eqref{T} and \eqref{Tm} we get \be\begin{split}
&|(T-T^{(m)})f(t)|\leq \int_0^b\left|\int_0^d
k(p,t,z)dp-\sum_{j=0}^m w_j^{(m)}k(p_j,t,z)
\right||f(z)|dz\\
&\leq \int_0^b\left|\frac{d^5}{180m^4}\max_{p\in[0,d]}(t+z)^4e^{-p(t+z)}\right||f(z)|dz\\
&=\int_0^b\frac{d^5}{180m^4}(t+z)^4|f(z)|dz\leq \frac{d^5}{180m^4}\left(\int_0^b(t+z)^8dz\right)^{1/2}\|f\|_{X_{0,b}}\\
&=\frac{d^5}{180m^4}\left[\frac{(t+b)^9-t^9}{9}\right]^{1/2}\|f\|_{X_{0,b}},
\end{split}\ee where the following upper bound for the error of the compound Simpson's rule was
used (see \cite[p.58]{PDPR84}): for $f\in C^{(4)}[x_0,x_{2l}],$
$x_0<x_{2l}$, \be
\left|\int_{x_0}^{x_{2l}}f(x)dx-\frac{h}{3}\left[f_0+4\sum_{j=1}^lf_{2(j-1)}+2\sum_{j=1}^{l-1}f_{2j}+f_{x_{2l}}\right]\right|\leq
R_l, \ee where \be f_j:=f(x_j),\quad x_j=x_0+jh,\
j=0,1,2,\hdots,2l,\ h=\frac{x_{2l}-x_0}{2l},\ee and \be
R_l=\frac{(x_{2l}-x_0)^5}{180(2l)^4}|f^{(4)}(\xi)|,\quad
x_0<\xi<x_{2l}. \ee This implies \be \|(T-T^{(m)})f\|_{X_{0,b}}\leq
\frac{d^5}{540
m^4}\left[\frac{(2b)^{10}-2b^{10}}{10}\right]^{1/2}\|f\|_{X_{0,b}}\leq
\frac{(2bd)^5}{540\sqrt{10} m^4}\|f\|_{X_{0,b}},
\ee so estimate \eqref{TTm} is obtained.\\
\lemref{lem1} is proved.
\end{proof}

\begin{lem}\label{lem2}
Let $0<a<a_0$,\be\label{rulem}
m=\kappa\left(\frac{a_0}{a}\right)^{1/4},\ \quad \kappa>0. \ee Then
\be\label{es2} \|T-T^{(m)}\|\leq \frac{(2bd)^5}{540\sqrt{10}
a_0\kappa^4}a, \ee where $T$ and $T^{(m)}$ are defined in \eqref{T}
and \eqref{Tm}, respectively.
\end{lem}
\begin{proof}
Inequality \eqref{es2} follows from estimate \eqref{TTm} and formula
\eqref{rulem}.
\end{proof}

\lemref{lem2} leads to an adaptive iterative scheme: \be\label{it2}
u_{n,m_n}(t)=qu_{n-1,m_{n-1}}+(1-q)T_{a_n,m_n}^{-1}\mathcal{L}_{m_n}^*F^{(m_n)},\quad
u_{0,m_0}(t)=0, \ee where $q\in(0,1)$, $a_n$ are defined in
\eqref{san}, $T_{a,m}$ is defined in \eqref{QTam},
$A_{m}\mathcal{L}$ is defined in \eqref{amkb}, and \be\label{Fm}
F^{(m)}:=\left(
                                             \begin{array}{c}
                                               F(p_0) \\
                                               F(p_1) \\
                                               \hdots \\
                                               F(p_{m}) \\
                                             \end{array}
                                           \right)\in \R^{m+1},
\ee $p_j$ are defined in \eqref{pj}. In the iterative scheme
\eqref{it2} we have used the finite-dimensional operator $T^{(m)}$
approximating the operator $T$. Convergence of the iterative scheme
\eqref{it2} to the solution $f$ of the equation $\mathcal{L}f=F$ is
established in the following lemma:
\begin{lem}\label{lem4}
Let $\mathcal{L}f=F$ and $u_{n,m_n}$ be defined in \eqref{it2}. If
$m_n$ are chosen by the rule \be\label{rulm}\begin{split}
m_n&=\left\lceil\left[\kappa
\left(\frac{a_0}{a_n}\right)^{1/4}\right]\right\rceil,\
a_n=qa_{n-1},\
q\in(0,1),\ \kappa, a_0>0,\\
\end{split}\ee where $\lceil[x]\rceil $ is the smallest even number not less than x, then \be \lim_{n\to \infty}\|f-u_{n,m_n}\|=0. \ee
\end{lem}
\begin{proof}
Consider the estimate \be\label{es3} \|f-u_{n,m_n}\|\leq
\|f-u_n\|+\|u_n-u_{n,m_n}\|:=I_1(n)+I_2(n), \ee where
$I_1(n):=\|f-u_n\|$ and $I_2(n):=\|u_n-u_{n,m_n}\|$. By
\thmref{thm1}, we get $I_1(n)\to 0$ as $n\to \infty.$ Let us prove
that $\lim_{n\to \infty}I_2(n)=0.$ Let $U_n:=u_n-u_{n,m_n}.$ Then,
from definitions \eqref{it1} and \eqref{it2}, we get \be
U_n=qU_{n-1}+(1-q)\left(T_{a_n}^{-1}\mathcal{L}^*F-T_{a_n,m_n}^{-1}\mathcal{L}_{m_n}^*F^{(m_n)}
\right),\quad U_0=0. \ee By induction we obtain \be
U_n=\sum_{j=0}^{n-1}\omega_j^{(n)}
\left(T_{a_{j+1}}^{-1}\mathcal{L}^*F-T_{a_{j+1},m_{j+1}}^{-1}(\mathcal{L}_{m_{j+1}})^*F^{(m_{j+1})}\right),
\ee where $\omega_j$ are defined in \eqref{omg}. Using the
identities $\mathcal{L}f=F$, $\mathcal{L}_mf=F^{(m)}$, \be
T_a^{-1}T=T_a^{-1}(T+aI-aI)=I-aT_a^{-1}, \ee \be
T_{a,m}^{-1}T^{(m)}=T_{a,m}^{-1}(T^{(m)}+aI-aI)=I-aT_{a,m}^{-1}, \ee
\be T_{a,m}^{-1}-T_a^{-1}=T_{a,m}^{-1}(T-T^{(m)})T_a^{-1}, \ee one
gets \be\begin{split}
U_n&=\sum_{j=0}^{n-1}\omega_j^{(n)}a_{j+1}\left(T_{a_{j+1},m_{j+1}}^{-1}-T_{a_{j+1}}^{-1}\right)f\\
&=\sum_{j=0}^{n-1}\omega_j^{(n)}a_{j+1}T_{a_{j+1},m_{j+1}}^{-1}\left(T-T^{(m_{j+1})}\right)T_{a_{j+1}}^{-1}f.
\end{split}\ee
This together with the rule \eqref{rulm}, estimate \eqref{eTam} and
\lemref{lem1} yield
\be\begin{split} \|U_n\|&\leq\sum_{j=0}^{n-1}\omega_j^{(n)}a_{j+1}\|T_{a_{j+1},m_{j+1}}^{-1}\|\|T-T^{(m_{j+1})}\|\|T_{a_{j+1}}^{-1}f\|\\
&\leq \frac{(2bd)^5}{540\sqrt{10}
a_0\kappa^4}\sum_{j=0}^{n-1}\omega_j^{(n)}a_{j+1}\|T_{a_{j+1}}^{-1}f\|.
\end{split} \ee Applying \lemref{lem12} and \lemref{lemq} with $g(a)=a\|T_a^{-1}f\|$,
we obtain $\lim_{n\to \infty}\|U_n\|=0.$\\
\lemref{lem4} is proved.
\end{proof}

\subsection{Noisy data}
When the data $F(p)$ are noisy, the approximate solution \eqref{aps}
is written as \be\label{fmdel} f_m^\dl(t)=\sum_{j=0}^m
w_j^{(m)}c_j^{(m,\dl)}e^{-p_j
t}=T_{a,m}^{-1}\mathcal{L}^*_mF_\dl^{(m)}, \ee where the
coefficients $c_j^{(m,\dl)}$ are obtained by solving the following
linear algebraic system: \be\label{LASdel}
Q_{a,m}c^{(m,\dl)}=F_\dl^{(m)}, \ee $Q_{a,m}$ is defined in
\eqref{QTam}, \be\label{Fmdel} c^{(m,\dl)}:=\left(
                                             \begin{array}{c}
                                               c_0^{(m,\dl)} \\
                                               c_1^{(m,\dl)} \\
                                               \hdots \\
                                               c_m^{(m,\dl)}\\
                                             \end{array}
                                           \right),\quad F_\dl^{(m)}:=\left(
                                             \begin{array}{c}
                                               F_\dl(p_0) \\
                                               F_\dl(p_1) \\
                                               \hdots \\
                                               F_\dl(p_{m}) \\
                                             \end{array}
                                           \right),
 \ee $w_j^{(m)}$ are defined in \eqref{wj}, and $p_j$ are defined in
\eqref{pj}.

To get the approximation solution of the function $f(t)$ with the
noisy data $F_\dl(p)$, we consider the following iterative
scheme:\be
\label{it3}u_{n,m_n}^\dl=qu_{n-1,m_{n-1}}^\dl+(1-q)T_{a_n,m_n}^{-1}\mathcal{L}_{m_n}^*F^{(m_n)}_\dl,\quad
u_{0,m_0}^\dl=0,\ee where $T_{a,m}$ is defined in \eqref{QTam},
$a_n$ are defined in \eqref{san}, $q\in(0,1)$,  $F_\dl^{(m)}$ is
defined in \eqref{Fmdel}, and $m_n$ are chosen by the rule
\eqref{rulm}. Let us assume that \be\label{as1}
F_\dl(p_j)=F(p_j)+\dl_j,\quad 0<|\dl_j|\leq \dl,\quad
j=0,1,2,\hdots,m, \ee where $\dl_j$ are random quantities generated
from some statistical distributions, e.g., the uniform distribution
on the interval $[-\dl,\dl]$, and $\dl$ is the noise level of the
data $F(p)$. It follows from assumption \eqref{as1}, definition
\eqref{wj}, \lemref{lemswj} and the inner product \eqref{ipWm} that
\be\label{es4}
\|F^{(m)}_\dl-F^{(m)}\|_{W^m}^2=\sum_{j=0}^mw_j^{(m)}\dl_j^2\leq
\dl^2\sum_{j=0}^mw_j^{(m)}=\dl^2d. \ee
\begin{lem}\label{lem5} Let $u_{n,m_n}$ and $u_{n,m_n}^\dl$ be defined in
\eqref{it2} and \eqref{it3}, respectively. Then \be\label{es6}
\|u_{n,m_n}-u_{n,m_n}^\dl\|\leq
\frac{\sqrt{d}\dl}{2\sqrt{a_n}}(1-q^n),\quad q\in(0,1), \ee where
$a_n$ are defined in \eqref{san}.
\end{lem}
\begin{proof}
Let $U_n^\dl:=u_{n,m_n}-u_{n,m_n}^\dl$. Then, from definitions
\eqref{it2} and \eqref{it3}, \be
U_n^\dl=qU_{n-1}^\dl+(1-q)T_{a_n,m_n}^{-1}\mathcal{L}_{m_n}^*(F^{(m_{n})}-F^{(m_{n})}_\dl),\quad
U_0^\dl=0. \ee By induction we obtain \be
U_n^\dl=\sum_{j=0}^{n-1}\omega_j^{(n)}T_{a_{j+1},m_{j+1}}^{-1}(\mathcal{L}_{m_{j+1}})^*(F^{(m_{j+1})}-F^{(m_{j+1})}_\dl),
\ee where $\omega_j^{(n)}$ are defined in \eqref{omg}. Using
estimates \eqref{es4} and inequality \eqref{eTLm}, one gets
\be\label{es7} \|U_n^\dl\|\leq
\sqrt{d}\sum_{j=0}^{n-1}\omega_j^{(n)}\frac{\dl}{2\sqrt{a_{j+1}}}\leq
\frac{\sqrt{d}\dl}{2\sqrt{a_n}}\sum_{j=0}^m\omega_j^{(n)}=\frac{\sqrt{d}\dl}{2\sqrt{a_n}}(1-q^n),
\ee where $\omega_j$ are defined in \eqref{omg}.\\ \lemref{lem5} is
proved.
\end{proof}

\begin{thm}\label{thm2}Suppose that conditions of \lemref{lem4} hold, and $n_\dl$ satisfies the following conditions:
\be\label{rel4} \lim_{\dl \to 0} n_{\dl}=\infty,\quad \lim_{\dl\to
0}\frac{\dl}{\sqrt{a_{n_{\dl}}}}=0. \ee Then\be\label{rel5}
\lim_{\dl\to 0}\|f-u_{n_\dl,m_{n_\dl}}^\dl\|=0. \ee
\end{thm}
\begin{proof}
Consider the estimate: \be \|f-u_{n_\dl,m_{n_\dl}}^\dl\|\leq
\|f-u_{n_\dl,m_{n_\dl}}\|+\|u_{n_\dl,m_{n_\dl}}-u_{n_\dl,m_{n_\dl}}^\dl\|.
\ee This together with \lemref{lem5} yield \be\label{es5}
\|f-u_{n_\dl,m_{n_\dl}}^\dl\|\leq
\|f-u_{n_\dl,m_{n_\dl}}\|+\frac{\sqrt{d}\dl}{2\sqrt{a_{n_\dl}}}(1-q^n).
\ee Applying relations \eqref{rel4} in estimate \eqref{es5}, one gets relation \eqref{rel5}.\\
\thmref{thm2} is proved.
\end{proof}
In the following subsection we propose a stopping rule which implies
relations \eqref{rel4}.

\subsection{Stopping rule}
In this subsection a stopping rule which yields relations
\eqref{rel4} in \thmref{thm2} is given.  We propose the stopping
rule \be\label{dp} G_{n_\dl,m_{n_\dl}}\leq
C\dl^\varepsilon<G_{n,m_n},\quad 1\leq n<n_\dl,\ C>\sqrt{d},\
\varepsilon \in(0,1), \ee where \be\label{Gn}
G_{n,m_n}=qG_{n-1,m_{n-1}}+(1-q)\|\mathcal{L}_{m_n}z^{(m_n,\dl)}-F^{(m_n)}_\dl\|_{W^{m_n}},\
G_{0,m_0}=0, \ee $\|\cdot\|_{W^m}$ is defined in \eqref{ipWm},
\be\label{zm}z^{(m,\dl)}:=\sum_{j=0}^{m} c_j^{(m,\dl)}
w_j^{(m)}e^{-p_jt}, \ee $w_j^{(m)}$ and $p_j$ are defined in
\eqref{wj} and \eqref{pj}, respectively, and $c_j^{(m,\dl)}$ are
obtained by solving linear algebraic system \eqref{LASdel}.

We observe that \be\begin{split}
\mathcal{L}_{m_n}z^{(m_n,\dl)}-F^{(m_n)}_\dl&=Q^{(m_n)}c^{(m_n,\dl)}-F^{(m_n)}_\dl\\
&=Q^{(m_n)}(a_nI+Q^{(m_n)})^{-1}F^{(m_n)}_\dl-F^{(m_n)}_\dl\\
&=(Q^{(m_n)}+a_nI-a_nI)(a_nI+Q^{(m_n)})^{-1}F^{(m_n)}_\dl-F^{(m_n)}_\dl\\
&=-a_n(a_nI+Q^{(m_n)})^{-1}F^{(m_n)}_\dl=-a_nc^{(m_n,\dl)}.
\end{split}\ee Thus, the sequence \eqref{Gn} can be written in the
following form \be\label{itG}
G_{n,m_n}=qG_{n-1,m_{n-1}}+(1-q)a_n\|c^{(m_n,\dl)}\|_{W^{m_n}},\
G_{0,m_0}=0, \ee  where $\|\cdot\|_{W^m}$ is defined in
\eqref{ipWm}, and $c^{(m,\dl)}$ solves the linear algebraic systems
\eqref{LASdel}.

It follows from estimates \eqref{es4}, \eqref{eQam} and
\eqref{eaQam} that \be\begin{split}\label{acm}
a_n\|c^{(m_n,\dl)}\|_{W^{m_n}}&=a_n\|(a_nI+Q^{(m_n)})^{-1}F_\dl^{(m_n)}\|_{W^{m_n}}\\
&\leq a_n\|(a_nI+Q^{(m_n)})^{-1}(F_\dl^{(m_n)}-F^{(m_n)})\|_{W^{m_n}}\\
&+a_n\|(a_nI+Q^{(m_n)})^{-1}F^{(m_n)}\|_{W^{m_n}}\\
&\leq \|F_\dl^{(m_n)}-F^{(m_n)}\|_{W^{m_n}}\\
&+a_n\|(a_nI+Q^{(m_n)})^{-1}\mathcal{L}_{m_n}f\|_{W^{m_n}}\\
&\leq \dl\sqrt{d}+\sqrt{a_n}\|f\|_{X_{0,b}}.
\end{split}\ee This together with \eqref{itG}
yield \be G_{n,m_n}\leq
qG_{n-1,m_{n-1}}+(1-q)\left(\dl\sqrt{d}+\sqrt{a_n}\|f\|_{X_{0,b}}\right),
\ee or \be\label{es9} G_{n,m_n}-\dl\sqrt{d}\leq
q(G_{n-1,m_{n-1}}-\dl\sqrt{d})+ (1-q)\sqrt{a_n}\|f\|_{X_{0,b}}.\ee

\begin{lem}\label{lem8} The sequence \eqref{itG} satisfies the
following estimate: \be\label{es10} G_{n,m_n}-\dl\sqrt{d}\leq
\frac{(1-q)\sqrt{a_n}\|f\|_{X_{0,b}}}{1-\sqrt{q}},\ee where $a_n$
are defined in \eqref{san}.
\end{lem}
\begin{proof}
Define \be \Psi_n:=G_{n,m_n}-\dl\sqrt{d} \ee and \be
\psi_n:=(1-q)\sqrt{a_n}\|f\|_{X_{0,b}}. \ee Then estimate
\eqref{es9} can be rewritten as \be \Psi_n\leq
q\Psi_{n-1}+\sqrt{q}\psi_{n-1}, \ee where the relation
$a_n=qa_{n-1}$ was used. Let us prove estimate \eqref{es10} by
induction. For $n=0$ we get \be \Psi_0=-\dl \sqrt{d}\leq
\frac{(1-q)\sqrt{a_0}\|f\|_{X_{0,b}}}{1-\sqrt{q}}. \ee Suppose
estimate \eqref{es10} is true for $0\leq n\leq k$. Then
\be\begin{split} \Psi_{k+1}&\leq q\Psi_{k}+\sqrt{q}\psi_k\leq
\frac{q}{1-\sqrt{q}}\psi_k+\sqrt{q}\psi_k
\\
&=\frac{\sqrt{q}}{1-\sqrt{q}}\psi_k=\frac{\sqrt{q}}{1-\sqrt{q}}\frac{\psi_k}{\psi_{k+1}}\psi_{k+1}\\
&=\frac{\sqrt{q}}{1-\sqrt{q}}\frac{\sqrt{a_k}}{\sqrt{a_{k+1}}}\psi_{k+1}=\frac{1}{1-\sqrt{q}}\psi_{k+1},\end{split}\ee
where the relation $a_{k+1}=qa_k$ was used.\\
\lemref{lem8} is proved.
\end{proof}
\begin{lem}\label{lem9}
Suppose \be\label{con1}G_{1,m_1}>\dl\sqrt{d},\ee where $G_{n,m_n}$
are defined in \eqref{itG}. Then there exist a unique integer
$n_\dl$, satisfying the stopping rule \eqref{dp} with $C>\sqrt{d}$.
\end{lem}
\begin{proof}
From \lemref{lem8} we get the estimate \be\label{es11} G_{n,m_n}\leq
\dl\sqrt{d}+ \frac{(1-q)\sqrt{a_n}\|f\|_{X_{0,b}}}{1-\sqrt{q}}, \ee
where $a_n$ are defined in \eqref{san}. Therefore, \be \limsup_{n\to
\infty} G_{n,m_n}\leq \dl \sqrt{d}, \ee where the relation
$\lim_{n\to \infty} a_n=0$ was used. This together with condition
\eqref{con1} yield the existence of the integer $n_\dl$. The
uniqueness of the integer $n_\dl$ follows from its definition.\\
\lemref{lem9} is proved.
\end{proof}
\begin{lem}\label{lem10}
Suppose conditions of \lemref{lem9} hold and $n_\dl$ is chosen by
the rule \eqref{dp}. Then \be \lim_{\dl\to
0}\frac{\dl}{\sqrt{a_{n_\dl}}}=0. \ee
\end{lem}
\begin{proof}
From the stopping rule \eqref{dp} and estimate \eqref{es11} we get
\be C\dl^\varepsilon\leq G_{n_\dl-1,m_{n_\dl-1}}\leq \dl \sqrt{d}+
\frac{(1-q)\sqrt{a_{n_\dl-1}}\|f\|_{X_{0,b}}}{1-\sqrt{q}}, \ee where
$C>\sqrt{d},$ $\varepsilon\in(0,1)$. This implies \be
\frac{\dl(C\dl^{\varepsilon-1}-\sqrt{d})}{\sqrt{a_{n_\dl-1}}}\leq
\frac{(1-q)\|f\|_{X_{0,b}}}{1-\sqrt{q}}, \ee so, for
$\varepsilon\in(0,1)$, and $a_{n_\dl}=qa_{n_\dl-1}$, one gets\be
\lim_{\dl\to 0}\frac{\dl}{\sqrt{a_{n_\dl}}}=\lim_{\dl\to
0}\frac{\dl}{\sqrt{q}\sqrt{a_{n_\dl-1}}}\leq \lim_{\dl\to 0}
\frac{(1-q)\dl^{1-\varepsilon}\|f\|_{X_{0,b}}}{(\sqrt{q}-q)(C-\dl^{1-\varepsilon}\sqrt{d})}=0.
\ee \lemref{lem10} is proved.
\end{proof}

\begin{lem}\label{lem11}
Consider the stopping rule \eqref{dp}, where the parameters $m_n$
are chosen by rule \eqref{rulm}. If $n_\dl$ is chosen by the rule
\eqref{dp} then \be \lim_{\dl\to 0}n_\dl=\infty.\ee
\end{lem}
\begin{proof}
From the stopping rule \eqref{dp} with the sequence $G_n$ defined in
\eqref{itG} one gets \be\begin{split}
&qC\dl^\varepsilon+(1-q)a_{n_\dl}\|c^{(m_{n_\dl},\dl)}\|_{W^{m_{n_\dl}}}\leq
qG_{n_\dl-1,m_{n_\dl-1}}\\
&+(1-q)a_{n_\dl}\|c^{(m_{n_\dl},\dl)}\|_{W^{m_{n_\dl}}}=G_{n_\dl,m_{n_\dl}}<C\dl^\varepsilon,
\end{split}\ee where $c^{(m,\dl)}$ is obtained by solving linear algebraic system \eqref{LASdel}. This implies \be
0<a_{n_\dl}\|c^{(m_{n_\dl},\dl)}\|_{W^{m_{n_\dl}}}\leq
C\dl^\varepsilon. \ee Thus, \be\label{rel6} \lim_{\dl\to
0}a_{n_\dl}\|c^{(m_{n_\dl},\dl)}\|_{W^{m_{n_\dl}}}=0. \ee
 If $F^{(m)}\neq 0$, then there exists a $\lambda_0^{(m)}>0$ such
that \be\label{Flo} E^{(m)}_{\lambda_0^{(m)}}F^{(m)}\neq 0,\quad
\langle
E^{(m)}_{\lambda_0}F^{(m)},F^{(m)}\rangle_{W^{m}}:=\xi^{(m)}>0, \ee
where $E^{(m)}_s$ is the resolution of the identity corresponding to
the operator $Q^{(m)}:=\mathcal{L}_m\mathcal{L}_m^*$. Let
$$h_m(\dl,\alpha):=\alpha^2\|Q_{m,\alpha}^{-1}F_\dl^{(m)}\|_{W^{m}}^2,\quad Q_{m,a}:=aI+Q^{(m)}.$$ For a
fixed number $a>0$ we obtain \be\begin{split}\label{hdl}
h_{m}(\dl,a)&=a^2\|Q_{m,a}^{-1}F_\dl^{(m)}\|_{W^{m}}^2\\
&=\int_0^\infty \frac{a^2}{(a+s)^2}d\langle
E^{(m)}_sF_\dl^{(m)},F_\dl^{(m)}\rangle_{W^{m}}\\
&\geq \int_0^{\lambda_0^{(m)}} \frac{a^2}{(a+s)^2}d\langle
E^{(m)}_sF_\dl^{(m)},F_\dl^{(m)}\rangle_{W^{m}}\\
&\geq\frac{a^2}{(a+\lambda_0)^2}\int_0^{\lambda_0^{(m)}} d\langle
E^{(m)}_sF_\dl^{(m)},F_\dl^{(m)}\rangle_{W^{m}}\\
&=\frac{a^2\|E^{(m)}_{\lambda_0^{(m)}}F_\dl^{(m)}\|_{W^{m}}^2}{(a+\lambda_0^{(m)})^2}.
\end{split}\ee
Since $E^{(m)}_{\lambda_0}$ is a continuous operator, and
$\|F^{(m)}-F^{(m)}_\dl\|_{W^m}<\sqrt{d}\dl$, it follows from
\eqref{Flo} that \be \lim_{\dl\to 0}\langle
E^{(m)}_{\lambda_0}F_\dl^{(m)},F_\dl^{(m)}\rangle_{W^m}=\langle
E^{(m)}_{\lambda_0}F^{(m)},F^{(m)}\rangle_{W^m}>0 .\ee Therefore,
for the fixed number $a>0$ we get \be\label{hc1} h_m(\dl,a)\geq
c_2>0\ee for all sufficiently small $\dl>0$, where $c_2$ is a
constant which does not depend on $\dl$. Suppose $\lim_{\dl\to
0}a_{n_\dl}\neq 0.$ Then there exists a subsequence $\dl_j\to 0$ as
$j\to \infty$, such that \be a_{n_{\dl_j}}\geq c_1>0, \ee and\be
0<m_{n_{\dl_j}}=\left\lceil
[\kappa(a_0/a_{n_{\dl_j}})^{1/4}]\right\rceil\leq \left\lceil
[\kappa(a_0/c_1)^{1/4}]\right\rceil:=c_3<\infty,\quad
\kappa,a_0>0,\ee where the rule \eqref{rulm} was used to obtain the
parameters $m_{n_{\dl_j}}$. This together with \eqref{Flo} and
\eqref{hc1} yield \be\begin{split} \lim_{j\to
\infty}h_{m_{n_{\dl_j}}}(\dl_j,a_{n_{\dl_j}})&\geq \lim_{j\to
\infty}
\frac{a_{n_{\dl_j}}^2\|E^{(m_{n_{\dl_j}})}_{\lambda_0^{(m_{n_{\dl_j}})}}F_{\dl_j}^{(m_{n_{\dl_j}})}\|_{W^{m_{n_{\dl_j}}}}^2}{(a_{n_{\dl_j}}+\lambda_0^{(m_{n_{\dl_j}})})^2}\\
& \geq \liminf_{j\to
\infty}\frac{c_1^2\|E^{(m_{n_{\dl_j}})}_{\lambda_0^{(m_{n_{\dl_j}})}}F^{(m_{n_{\dl_j}})}\|_{W^{m_{n_{\dl_j}}}}^2}{(c_1+\lambda_0^{(m_{n_{\dl_j}})})^2}>0.\end{split}\ee
This contradicts relation \eqref{rel6}. Thus, $\lim_{\dl\to0}
a_{n_\dl}=\lim_{\dl\to 0} a_0q^{n_\dl}=0,$ i.e., $\lim_{\dl\to
0}n_\dl=\infty.$\\ \lemref{lem11} is proved.
\end{proof}

It follows from \lemref{lem10} and \lemref{lem11} that the stopping
rule \eqref{dp} yields the relations \eqref{rel4}. We have proved
the following theorem:

\begin{thm}\label{thm3}
Suppose all the assumptions of \thmref{thm2} hold, $m_n$ are chosen
by the rule \eqref{rulm}, $n_\dl$ is chosen by the rule \eqref{dp}
and $G_{1,m_1}>C\dl,$ where $G_{n,m_n}$ are defined in \eqref{itG},
then \be\label{rel7} \lim_{\dl\to 0}\|f-u_{n_\dl,m_{n_\dl}}^\dl\|=0.
\ee
\end{thm}
\subsection{The algorithm}
Let us formulate the algorithm for obtaining the approximate
solution $f_m^\dl$:
\begin{itemize}
\item[(1)] The data $F_\dl(p)$ on the interval $[0,d]$, $d>0$, the support of the function $f(t)$, and the noise level $\dl$;
\item[(2)] initialization : choose the parameters $\kappa>0$, $a_0>0$, $q\in(0,1)$, $\varepsilon\in(0,1)$, $C>\sqrt{d}$, and set $u_{0,m_0}^\dl=0$, $G_0=0$, $n=1$;
\item[(3)] iterate, starting with $n=1$, and stop when condition
\eqref{stopit} ( see below) holds,
\begin{itemize}
\item[(a)]$ a_n=a_0q^n$,
\item[(b)]choose $m_n$ by the rule \eqref{rulm},
\item[(c)] construct the vector $F^{(m_n)}_\dl$:
\be (F^{(m_n)}_\dl)_l=F_\dl(p_l),\quad p_l=lh,\ h=d/m_n,\
l=0,1,\hdots,m, \ee
\item[(d)] construct the matrices $H_{m_n}$ and $D_{m_n}$:
\be
(H_{m_n})_{ij}:=\int_0^be^{-(p_i+p_j)t}dt=\frac{1-e^{-b(p_i+p_j)}}{p_i+p_j},\quad
i,j=1,2,3,\hdots,m_n\ee\be (D_{m_n})_{ij}=\left\{
                 \begin{array}{ll}
                   w_i^{(m_n)}, & \hbox{$i=j$;} \\
                   0, & \hbox{otherwise,}
                 \end{array}
               \right. \ee where $w_j^{(m)}$ are defined in
               \eqref{wj},
\item[(e)] solve the following linear algebraic
systems:\be\label{LAS1}(a_nI+H_{m_n}D_{m_n})c^{(m_n,\dl)}=F^{(m_n)}_\dl,\ee
where $(c^{(m_n,\dl)})_i=c_i^{(m_n,\dl)}$,
\item[(f)]update the coefficient $c_j^{(m_n,\dl)}$ of the approximate solution $u_{n,m_n}^\dl(t)$ defined in \eqref{fmdel} by the
iterative formula:
\be\label{uap}u_{n,m_n}^\dl(t)=qu_{n-1,m_{n-1}}^\dl(t)+(1-q)\sum_{j=1}^{m_n}
c^{(m_n,\dl)}w_j^{(m_n)}e^{-p_jt} ,\ee where \be
u_{0,m_0}^\dl(t)=0. \ee
\end{itemize}
Stop when for the first time the inequality
\be\label{stopit}G_{n,m_n}=qG_{n-1,m_{n-1}}+a_n\|c^{(m_n,\dl)}\|_{W^{m_n}}\leq
C\dl^\varepsilon\ee holds, and get the approximation
$f^\dl(t)=u_{n_\dl,m_{n_\dl}}^\dl(t)$ of the function $f(t)$ by
formula \eqref{uap}.
\end{itemize}

\section{Numerical experiments}

\subsection{The parameters $\kappa$, $a_0$, $d$}
From definition \eqref{san} and the rule \eqref{rulm} we conclude
that $m_{n}\to \infty$ as $a_n\to 0.$ Therefore, one needs to
control the value of the parameter $m_{n}$ so that it will not grow
too fast as $a_n$ decreases. The role of the parameter $\kappa$ in
\eqref{rulm} is to control the value of the parameter $m_{n}$ so
that the value of the parameter $m_n$ will not be too large. Since
for sufficiently small noise level $\dl$, namely
$\dl\in(10^{-16},10^{-6}]$, the regularization parameter
$a_{n_\dl}$, obtained by the stopping rule \eqref{dp}, is at most
$O(10^{-9})$, we suggest to choose $\kappa$ in the interval $(0,1]$.
For the noise level $\dl\in(10^{-6},10^{-2}]$ one can choose
$\kappa\in(1,3]$. To reduce the number of iterations we suggest to
choose the geometric sequence $a_n=a_0\dl^{\alpha n}$, where
$a_0\in[0.1,0.2]$ and $\alpha\in[0.5,0.9].$ One may assume without
loss of generality that $b=1$, because a scaling transformation
reduces the integral over $(0,b)$ to the integral over $(0,1)$. We
have assumed that the data $F(p)$ are defined on the interval
$J:=[0, d]$. In the case the interval $J=[d_1,d]$, $0<d_1<d$, the
constant $d$ in estimates \eqref{es2}, \eqref{es4}, \eqref{es6},
\eqref{es7}, \eqref{es9}, \eqref{es10}, and \eqref{es11} are
replaced with the constant $d-d_1$. If $b=1$, i.e., $f(t)=0$ for
$t>1$, then one has to take $d$ not too large. Indeed, if $f(t)=0$
for $t>1$, then an integration by parts yields:
$F(p)=[f(0)-e^{-p}f(1)]/p +O(1/p^2),$ $p\to \infty.$ If the data are
noisy, and the noise level is $\delta$, then the data becomes
indistinguishable from noise for $p=O(1/\delta)$. Therefore it is
useless to keep the data $F_\dl(p)$ for $d>O(1/\delta)$. In practice
one may get a satisfactory accuracy of inversion by the method,
proposed in this paper, when one uses the data with $d\in [1,20]$
when $\delta\leq 10^{-2}$. In all the numerical examples we have
used $d=5$. Given the interval $[0,d]$, the proposed method
generates automatically the discrete data $F_\dl(p_j)$,
$j=0,1,2,\hdots,m$, over the interval $[0,d]$ which are needed to
get the approximation of the function $f(t)$.

\subsection{Experiments}
To test the proposed method we consider some examples proposed in
\cite{AYAGR00}, \cite{KSC76}, \cite{SDMR05}, \cite{LDAM02},
\cite{BDBM79}, \cite{HDJA68}, \cite{FMSS}, \cite{VVK06},
\cite{MCDG07} and \cite{JWEP78}. To illustrate the numerical
stability of the proposed method with respect to the noise, we use
the noisy data $F_\dl(p)$ with various noise levels $\dl=10^{-2},$
$\dl=10^{-4}$ and $\dl=10^{-6}$. The random quantities $\dl_j$ in
\eqref{as1} are obtained from the uniform probability density
function over the interval $[-\dl,\dl]$.
 In examples 1-12 we choose the
value of the parameters as follows: $a_n=0.1q^n$, $q=\dl^{1/2}$ and
$d=5$. The parameter $\kappa=1$ is used for the noise levels
$\dl=10^{-2}$ and $\dl=10^{-4}$. When $\dl=10^{-6}$ we choose
$\kappa=0.3$ so that the value of the parameters $m_n$ are not very
large, namely $m_n\leq 300$. Therefore, the computation time for
solving linear algebraic system \eqref{LAS1} can be reduced
significantly. We assume that the support of the function $f(t)$ is
in the interval $[0,b]$ with $b=10$. In the stopping rule \eqref{dp}
the following parameters are used: $C=\sqrt{d}+0.01$,
$\varepsilon=0.99$. In example 13 the function $f(t)=e^{-t}$ is used
to test the applicability of the proposed method to functions
without compact support. The results are given in Table 13 and
Figure 13.

For a comparison with the exact solutions we use the mean absolute
error: \be
MAE:=\left[\frac{\sum_{j=1}^{100}(f(t_i)-f^\dl_{m_{n_\dl}}(t_i))^2}{100}\right]^{1/2},\
t_j=0.01+0.1(j-1),\ j=1,\hdots,100, \ee where $f(t)$ is the exact
solution and $f^\dl_{m_{n_\dl}}(t)$ is the approximate solution. The
computation time (CPU time) for obtaining the approximation of
$f(t)$, the number of iterations (Iter.), and the parameters
$m_{n_\dl}$ and $a_{n_\dl}$ generated by the proposed method are
given in each
experiment (see Tables 1-12). All the calculations are done in double precision generated by MATLAB.\\

\begin{itemize}
\item \textbf{\textit{ Example 1. (see \cite{FMSS})}}\\
$$f_1(t)=\left\{
                                                                   \begin{array}{ll}
                                                                     1, & \hbox{$1/2\leq t\leq 3/2$,} \\
                                                                     0, & \hbox{$otherwise$,}
                                                                   \end{array}
                                                                 \right.\quad F_1(p)=\left\{
                                                                                       \begin{array}{ll}
                                                                                         1, & \hbox{$p=0$,} \\
                                                                                         \frac{e^{-p/2}-e^{-3p/2}}{p}, & \hbox{$p>0$.}
                                                                                       \end{array}
                                                                                     \right.
$$
\begin{figure}[htbp] \centering
\begin{tabular}{c}
\epsfig{file=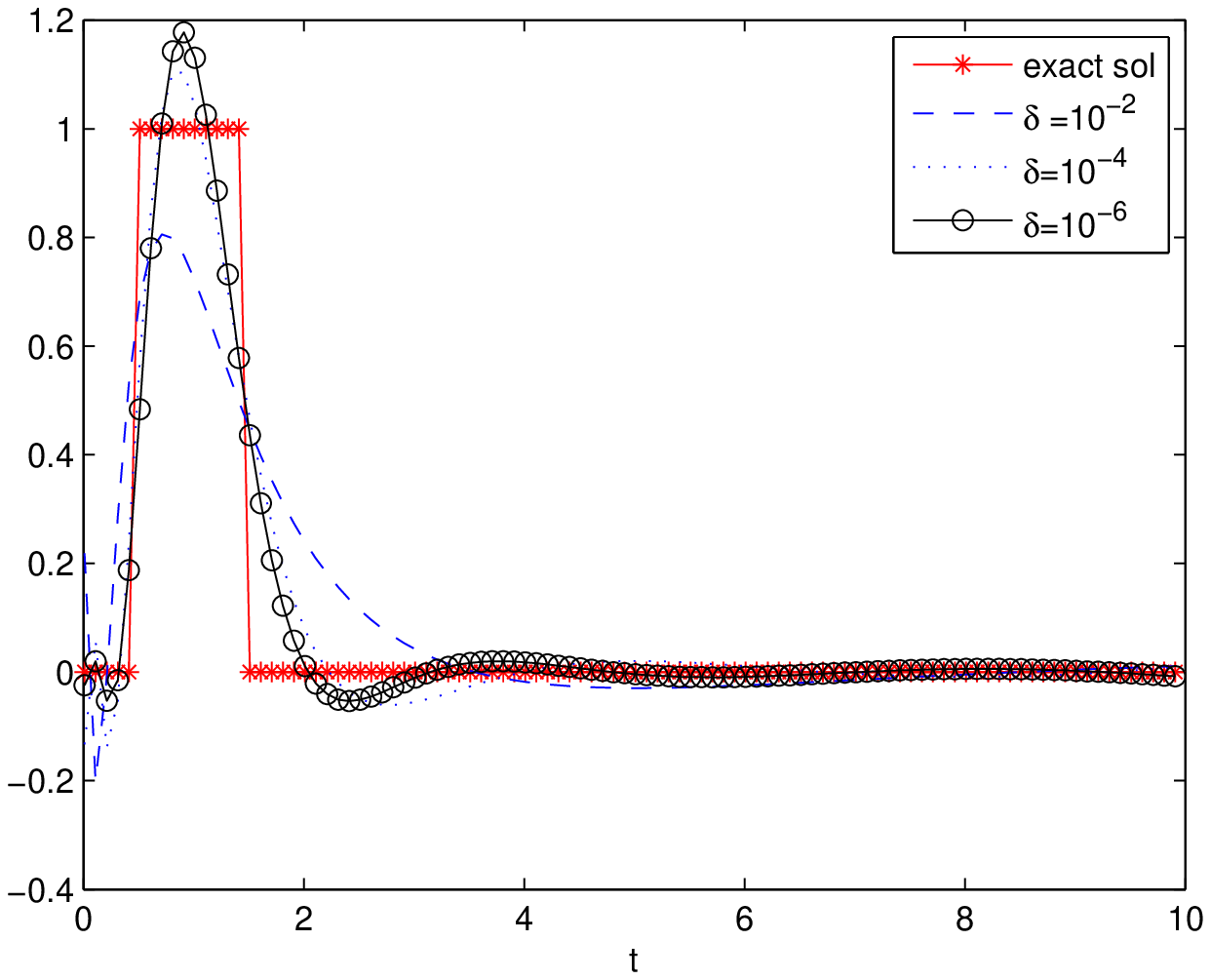,height=4cm,width=8cm,clip=} \\
\end{tabular}
\caption{Example 1: the stability of the approximate
solution}\label{fig:f1}
\end{figure}
\begin{table}[htbp]\caption{Example 1.}\label{tab1}
\newcommand{\m}{\hphantom{$-$}}
\renewcommand{\arraystretch}{1.2} %\begin{center}{\scriptsize
\begin{tabular}{llllll}
\hline
$\dl$&$MAE$&$m_{n_\dl}$&$Iter.$&CPU time(second)&$a_{n_\dl}$\\

\hline
$   1.00\times 10^{-2}    $&$ 9.62\times 10^{-2}    $&$ 30  $&$ 3   $&$ 3.13\times 10^{-2}    $&$ 2.00\times 10^{-3}    $\\
$   1.00\times 10^{-4}    $&$ 5.99\times 10^{-2}    $&$ 32  $&$ 4   $&$ 6.25\times 10^{-2}    $&$ 2.00\times 10^{-7}    $\\
$   1.00\times 10^{-6}    $&$ 4.74\times 10^{-2}    $&$ 54 $&$ 5   $&$ 3.28\times 10^{-1}    $&$ 2.00\times 10^{-10}    $\\
\hline \hline
\end{tabular}%}\end{center}
\end{table}

The reconstruction of the exact solution for different values of
the noise level $\dl$ is shown in Figure 1. When the noise level
$\dl=10^{-6},$ our result is comparable with the double
precision results shown in \cite{FMSS}. The proposed method is
stable with respect to the noise $\dl$ as shown in Table 1.

\item \textbf{\textit{ Example 2. (see \cite{LDAM02}, \cite{FMSS} )} }\\
$$f_2(t)=\left\{
                                       \begin{array}{ll}
                                         1/2, & \hbox{$t=1$,} \\
                                         1, & \hbox{$1<t<10$,}\\
                                         0, & \hbox{elsewhere,}
                                       \end{array}
                                     \right.\quad F_2(p)=\left\{
                                                           \begin{array}{ll}
                                                             9, & \hbox{$p=0$,} \\
                                                             \frac{e^{-p}-e^{-10p}}{p} , & \hbox{$p>0$.}
                                                           \end{array}
                                                         \right.
$$ \begin{figure}[htp] \centering
\begin{tabular}{c}
\epsfig{file=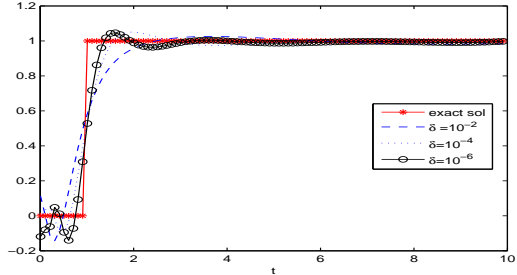,height=4cm,width=8cm,clip=} \\
\end{tabular}
\caption{Example 2: the stability of the approximate
solution}\label{fig:f2}
\end{figure}
\begin{table}[htbp]\caption{Example 2.}\label{tab2}
\newcommand{\m}{\hphantom{$-$}}
\renewcommand{\arraystretch}{1.2} %\begin{center}{\scriptsize
\begin{tabular}{cccccc}
\hline
$\dl$&$MAE$&$m_{n_\dl}$&$Iter.$&CPU time (seconds)&$a_{n_\dl}$\\
\hline \hline
$   1.00\times 10^{-2}    $&$ 1.09\times 10^{-1}    $&$ 30  $&$ 2   $&$ 3.13\times 10^{-2}    $&$ 2.00\times 10^{-3}    $\\
$   1.00\times 10^{-4}    $&$ 8.47\times 10^{-2}    $&$ 32  $&$ 3   $&$ 6.25\times 10^{-2}    $&$ 2.00\times 10^{-6}    $\\
$   1.00\times 10^{-6}    $&$ 7.41\times 10^{-2}    $&$ 54 $&$ 5   $&$ 4.38\times 10^{-1}    $&$ 2.00\times 10^{-12}    $\\
\hline \hline
\end{tabular}%}\end{center}
\end{table}

The reconstruction of the function $f_2(t)$ is plotted in Figure
2. In \cite{FMSS} a high accuracy result is given by means of
the multiple precision. But, as reported in \cite{FMSS}, to get
such high accuracy results, it takes 7 hours. From Table 2 and
Figure 2 we can see that the proposed method yields stable
solution with respect to the noise level $\dl$. The
reconstruction of the exact solution obtained by the proposed
method is better than the reconstruction shown in \cite{LDAM02}.
The result is comparable with the double precision results given
in \cite{FMSS}. For $\dl=10^{-6}$ and $\kappa=0.3$ the value of
the parameter $m_{n_\dl}$ is bounded by the constant $54$.

\item \textbf{\textit{ Example 3. (see \cite{AYAGR00}, \cite{LDAM02}, \cite{BDBM79}, \cite{MCDG07}, \cite{JWEP78})}}\\
$$f_3(t)=\left\{
           \begin{array}{ll}
             te^{-t}, & \hbox{$0\leq t<10$,} \\
             0, & \hbox{otherwise,}
           \end{array}
         \right.\quad
   F_3(p)=\frac{1-e^{-(p+1)10}}{(p+1)^2}-\frac{10e^{-(p+1)10}}{p+1}.$$
\begin{figure}[htbp] \centering
\begin{tabular}{c}
\epsfig{file=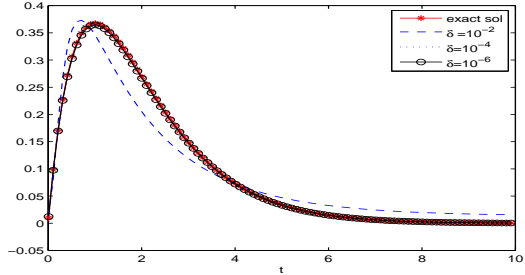,height=4cm,width=8cm,clip=} \\
\end{tabular}
\caption{Example 3: the stability of the approximate
solution}\label{fig:f3}
\end{figure}

\begin{table}[htbp]\caption{Example 3.}\label{tab3}
\newcommand{\m}{\hphantom{$-$}}
\renewcommand{\arraystretch}{1.2} %\begin{center}{\scriptsize
\begin{tabular}{llllll}
\hline
$\dl$&$MAE$&$m_{n_\dl}$&$Iter.$&CPU time (seconds)&$a_{n_\dl}$\\
\hline \hline
$   1.00\times 10^{-2}    $&$ 2.42\times 10^{-2}    $&$ 30  $&$ 2   $&$ 3.13\times 10^{-2}    $&$ 2.00\times 10^{-3}    $\\
$   1.00\times 10^{-4}    $&$ 1.08\times 10^{-3}    $&$ 30  $&$ 3   $&$ 3.13\times 10^{-2}    $&$ 2.00\times 10^{-6}    $\\
$   1.00\times 10^{-6}    $&$ 4.02\times 10^{-4}    $&$ 30 $&$ 4   $&$ 4.69\times 10^{-2}    $&$ 2.00\times 10^{-9}    $\\
\hline \hline
\end{tabular}%}\end{center}
\end{table}

We get an excellent agreement between the approximate solution
and the exact solution when the noise level $\dl=10^{-4}$ and
$10^{-6}$ as shown in Figure 3. The results obtained by the
proposed method are better than the results given in
\cite{LDAM02}. The mean absolute error $MAE$ decreases as the
noise level decreases which shows the stability of the proposed
method. Our results are more stable with respect to the noise
$\dl$ than the results presented in \cite{JWEP78}. The value of
the parameter $m_{n_\dl}$ is bounded by the constant $30$ when
the noise level $\dl=10^{-6}$ and $\kappa=0.3$.

\item \textbf{\textit{ Example 4. (see \cite{LDAM02}, \cite{FMSS})}}\\
\bee\begin{split}f_4(t)&=\left\{
           \begin{array}{ll}
             1-e^{-0.5t}, & \hbox{$0\leq t<10$,} \\
             0, & \hbox{elsewhere.}
           \end{array}
         \right. \\ F_4(p)&=\left\{
                           \begin{array}{ll}
                             8+2e^{-5}, & \hbox{$p=0$,} \\
                             \frac{1-e^{-10p}}{p}-\frac{1-e^{-(p+1/2)10}}{p+0.5}, & \hbox{$p>0$.}
                           \end{array}
                         \right.\end{split}\eee\begin{figure}[htbp]
                         \centering
\begin{tabular}{c}
\epsfig{file=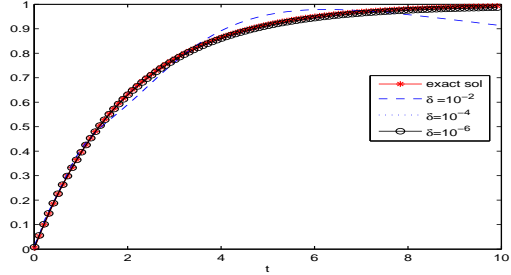,height=4cm,width=8cm,clip=} \\
\end{tabular}
\caption{Example 4: the stability of the approximate
solution}\label{fig:f4}
\end{figure}

As in our example 3 when the noise $\dl=10^{-4}$ and $10^{-6}$
are used, we get a satisfactory agreement between the
approximate solution and the exact solution. Table 4 gives the
results of the stability of the proposed method with respect to
the noise level $\dl$. Moreover, the reconstruction of the
function $f_4(t)$ obtained by the proposed method is better than
the reconstruction of $f_4(t)$ shown in \cite{LDAM02}, and is
comparable with the double precision reconstruction obtained in
\cite{FMSS}.
\begin{table}[htbp]\caption{Example 4.}\label{tab4}
\newcommand{\m}{\hphantom{$-$}}
\begin{tabular}{cccccc}
\hline
$\dl$&$MAE$&$m_{n_\dl}$&$Iter.$&CPU time (seconds)&$a_{n_\dl}$\\
\hline \hline
$   1.00\times 10^{-2}    $&$ 1.59\times 10^{-2}    $&$ 30  $&$ 2   $&$ 3.13\times 10^{-2}    $&$ 2.00\times 10^{-3}    $\\
$   1.00\times 10^{-4}    $&$ 8.26\times 10^{-4}    $&$ 30  $&$ 3   $&$ 9.400\times 10^{-2}    $&$ 2.00\times 10^{-6}    $\\
$   1.00\times 10^{-6}    $&$ 1.24\times 10^{-4}    $&$ 30 $&$ 4   $&$ 1.250\times 10^{-1}    $&$ 2.00\times 10^{-9}    $\\

\hline \hline
\end{tabular}%}\end{center}
\end{table}

In this example when $\dl=10^{-6}$ and $\kappa=0.3$ the value of
the parameter $m_{n_\dl}$ is bounded by the constant $109$ as
shown in Table 4.

\item \textbf{\textit{ Example 5. (see \cite{KSC76}, \cite{LDAM02}, \cite{HDJA68})}}\\
\bee\begin{split} f_5(t)&=2/\sqrt{3}e^{-t/2}\sin(t\sqrt{3}/2) \\
F_5(p)&=\frac{1-\cos(10\sqrt{3}/2)e^{-10(p+0.5)}}{[(p+0.5)^2+3/4]}-\frac{2(p+0.5)e^{-10(p+0.5)}\sin(10\sqrt{3}/2)}{\sqrt{3}[(p+0.5)^2+3/4]}.
\end{split}\eee\begin{figure}[htpb] \centering
\begin{tabular}{c}
\epsfig{file=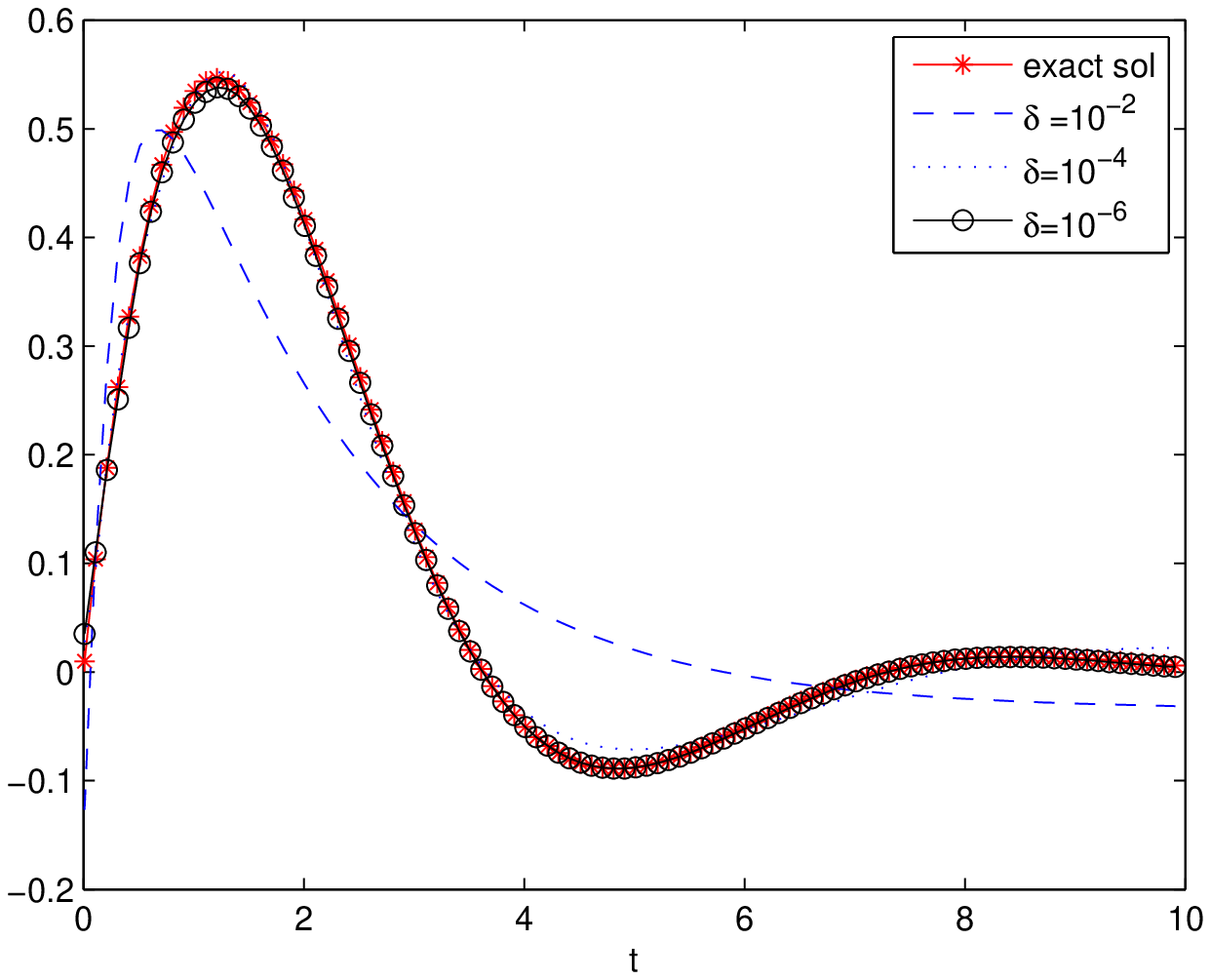,height=4cm,width=8cm,clip=} \\
\end{tabular}
\caption{Example 5: the stability of the approximate
solution}\label{fig:f5}
\end{figure}

\begin{table}[htpb]\caption{Example 5.}\label{tab5}
\newcommand{\m}{\hphantom{$-$}}
\renewcommand{\arraystretch}{1.2} %\begin{center}{\scriptsize
\begin{tabular}{llllll}
\hline
$\dl$&$MAE$&$m_{n_\dl}$&$Iter.$&CPU time (seconds)&$a_{n_\dl}$\\

\hline \hline
$   1.00\times 10^{-2}    $&$ 4.26\times 10^{-2}    $&$ 30  $&$ 3   $&$ 6.300\times 10^{-2}   $&$ 2.00\times 10^{-3}    $\\
$   1.00\times 10^{-4}    $&$ 1.25\times 10^{-2}    $&$ 30  $&$ 3   $&$ 9.38\times 10^{-2}    $&$ 2.00\times 10^{-6}    $\\
$   1.00\times 10^{-6}    $&$ 1.86\times 10^{-3}    $&$ 54 $&$ 4   $&$ 3.13\times 10^{-2}    $&$ 2.00\times 10^{-9}    $\\
\hline \hline
\end{tabular}%}\end{center}
\end{table}

This is an example of the damped sine function. In \cite{KSC76}
and \cite{HDJA68} the knowledge of the exact data $F(p)$ in the
complex plane is required to get the approximate solution. Here
we only use the knowledge of the discrete perturbed data
$F_\dl(p_j)$, $j=0,1,2,\hdots,m,$  and get a satisfactory result
which is comparable with the results given in \cite{KSC76} and
\cite{HDJA68} when the level noise $\dl=10^{-6}$. The
reconstruction of the exact solution $f_5(t)$ obtained by our
method is better than this of the method given in \cite{LDAM02}.
Moreover, our method yields stable solution with respect to the
noise level $\dl$ as shown in Figure 5 and Table 5 show. In this
example when $\kappa=0.3$ the value of the parameter $m_{n_\dl}$
is bounded by $54$ for the noise level $\dl=10^{-6}$ (see Table
5).

\item \textbf{\textit{ Example 6. (see \cite{FMSS})}}\\
\bee\begin{split}f_6(t)&=\left\{
           \begin{array}{ll}
             t, & \hbox{$0\leq t<1$,} \\
             3/2-t/2, & \hbox{$1\leq t<3$,} \\
             0, & \hbox{elsewhere.}
           \end{array}
         \right.\\
 F_6(p)&=\left\{
                               \begin{array}{ll}
                                 3/2, & \hbox{$p=0$,} \\
                                 \frac{1-e^{-p}(1+p)}{p^2}+\frac{e^{-3p}+e^{2p}(2p-1)}{2p^2}, & \hbox{$p>0$.}
                               \end{array}
                             \right.
  \end{split}\eee\begin{figure}[htpb] \centering
\begin{tabular}{c}
\epsfig{file=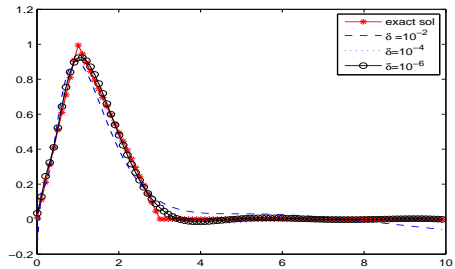,height=4cm,width=7cm,clip=} \\
\end{tabular}
\caption{Example 6: the stability of the approximate
solution}\label{fig:f6}
\end{figure}
\begin{table}[htpb]\caption{Example 6.}\label{tab6}
\newcommand{\m}{\hphantom{$-$}}
\renewcommand{\arraystretch}{1.2} %\begin{center}{\scriptsize
\begin{tabular}{llllll}
\hline
$\dl$&$MAE$&$m_{n_\dl}$&$Iter.$&CPU time (seconds)&$a_{n_\dl}$\\
 \hline \hline
$   1.00\times 10^{-2}    $&$ 4.19\times 10^{-2}    $&$ 30  $&$ 2   $&$ 4.700\times 10^{-2}   $&$ 2.00\times 10^{-3}    $\\
$   1.00\times 10^{-4}    $&$ 1.64\times 10^{-2}    $&$ 32  $&$ 3   $&$ 9.38\times 10^{-2}    $&$ 2.00\times 10^{-6}    $\\
$   1.00\times 10^{-6}    $&$ 1.22\times 10^{-2}    $&$ 54 $&$ 4   $&$ 3.13\times 10^{-2}    $&$ 2.00\times 10^{-9}    $\\
\hline \hline
\end{tabular}%}\end{center}
\end{table}

Example 6 represents a class of piecewise continuous functions.
>From Figure 6 the value of the exact solution at the points
where the function is not differentiable can not be well
approximated for the given levels of noise by the proposed
method. When the noise level $\dl=10^{-6}$, our result is
comparable with the results given in \cite{FMSS}. Table 6
reports the stability of the proposed method with respect to the
noise $\dl$. It is shown in Table 6 that the value of the
parameter $m$ generated by the proposed adaptive stopping rule
is bounded by the constant 54 for the noise level $\dl=10^{-6}$
and $\kappa=0.3$ which gives a relatively small computation
time.
\item \textbf{\textit{ Example 7. (see \cite{FMSS})}}\\
\bee\begin{split}f_7(t)&=\left\{
           \begin{array}{ll}
             -te^{-t}-e^{-t}+1, & \hbox{$0\leq t<1$,} \\
             1-2e^{-1}, & \hbox{$1\leq t<10$,}\\
             0, &\hbox{elsewhere,}
           \end{array}
         \right.
  \\
F_7(p)&=\left\{
          \begin{array}{ll}
            3/e-1+9(1-2/e), & \hbox{$p=0$,} \\
            e^{-1-p}\frac{e^{1+p}-e(1+p)^2+p(3+2p)}{p(p+1)^2}+(e-2)e^{-1-p-10p}\frac{e^{10p}-e^{p}}{p}, & \hbox{$p>0$.}
          \end{array}
        \right.\end{split}\eee

\begin{figure}[htpb] \centering
\begin{tabular}{c}
\epsfig{file=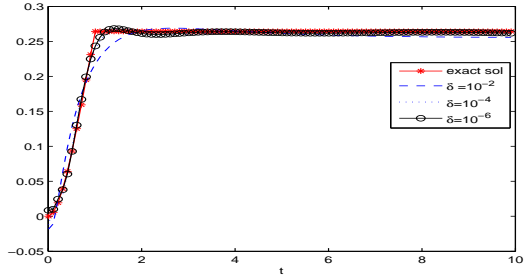,height=4cm,width=8cm,clip=} \\
\end{tabular}
\caption{Example 7: the stability of the approximate
solution}\label{fig:f7}
\end{figure}

\begin{table}[htpb]\caption{Example 7.}\label{tab7}
\newcommand{\m}{\hphantom{$-$}}
\renewcommand{\arraystretch}{1.2} %\begin{center}{\scriptsize
\begin{tabular}{llllll}
\hline
$\dl$&$MAE$&$m_{n_\dl}$&$Iter.$&CPU time (seconds)&$a_{n_\dl}$\\

\hline \hline
$   1.00\times 10^{-2}    $&$ 1.52\times 10^{-2}    $&$ 30  $&$ 2   $&$ 4.600\times 10^{-2}   $&$ 2.00\times 10^{-3}    $\\
$   1.00\times 10^{-4}    $&$ 2.60\times 10^{-3}    $&$ 30  $&$ 3   $&$ 9.38\times 10^{-2}    $&$ 2.00\times 10^{-6}    $\\
$   1.00\times 10^{-6}    $&$ 2.02\times 10^{-3}    $&$ 30 $&$ 4   $&$ 3.13\times 10^{-2}    $&$ 2.00\times 10^{-9}    $\\

\hline \hline
\end{tabular}%}\end{center}
\end{table}

When the noise level $\dl=10^{-4}$ and $\dl=10^{-6}$, we get
        numerical results which are comparable with the double
        precision results given in \cite{FMSS}. Figure 7 and
        Table 7 show the stability of the proposed method for
        decreasing $\dl$.
\item \textbf{\textit{ Example 8. (see \cite{SDMR05}, \cite{LDAM02})}}\\
\bee\begin{split}f_8(t)&=\left\{
           \begin{array}{ll}
             4t^2e^{-2t}, & \hbox{$0\leq t<10$,} \\
             0, & \hbox{elsewhere.}
           \end{array}
         \right.\quad\\
 F_8(p)&=\frac{8 + 4e^{-10(2 + p)}[-2 - 20(2 + p)-100(2-p)^2
 ]}{(2 + p)^3} .\end{split}\eee
\begin{figure}[htbp] \centering
\begin{tabular}{c}
\epsfig{file=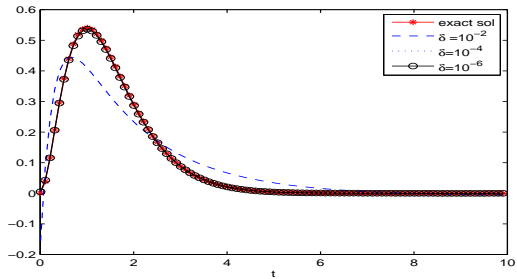,height=4cm,width=8cm,clip=} \\
\end{tabular}
\caption{Example 8: the stability of the approximate
solution}\label{fig:f8}
\end{figure}

The results of this example are similar to the results of
         Example 3. The exact solution can be well reconstructed
         by the approximate solution obtained by our method at
        the levels noise $\dl=10^{-4}$ and $\dl=10^{-6}$
         (see Figure 8). Table 8 shows that the MAE decreases as
         the noise level decreases which shows the stability of
         the proposed method with respect to the noise. In all
         the levels of noise $\dl$ the computation time of the
         proposed method in obtaining the approximate solution
         are relatively small. We get better reconstruction
         results than the results shown in \cite{LDAM02}. Our
         results are comparable with the results given in
         \cite{SDMR05}.

\begin{table}[htbp]\caption{Example 8.}\label{tab8}
\newcommand{\m}{\hphantom{$-$}}
\renewcommand{\arraystretch}{1.2} %\begin{center}{\scriptsize
\begin{tabular}{llllll}
\hline
$\dl$&$MAE$&$m_{n_\dl}$&$Iter.$&CPU time (seconds)&$a_{n_\dl}$\\

\hline \hline
$   1.00\times 10^{-2}    $&$ 2.74\times 10^{-2}    $&$ 30  $&$ 2   $&$ 1.100\times 10^{-2}    $&$ 2.00\times 10^{-3}    $\\
$   1.00\times 10^{-4}    $&$ 3.58\times 10^{-3}    $&$ 30  $&$ 3   $&$ 3.13\times 10^{-2}    $&$ 2.00\times 10^{-6}    $\\
$   1.00\times 10^{-6}    $&$ 5.04\times 10^{-4}    $&$ 30 $&$ 4   $&$ 4.69\times 10^{-2}    $&$ 2.00\times 10^{-9}    $\\
\hline \hline
\end{tabular}%}\end{center}
\end{table}

\item \textbf{\textit{ Example 9. (see \cite{MCDG07})}}\\
\bee\begin{split}f_9(t)&=\left\{
           \begin{array}{ll}
             5-t, & \hbox{$0\leq t<5$,} \\
             0, & \hbox{elsewhere,}
           \end{array}
         \right.\\ F_9(p)&=\left\{
                               \begin{array}{ll}
                                 25/2, & \hbox{$p=0$,} \\
                                 \frac{e^{-5p}+5p-1}{p^2}, & \hbox{$p>0$.}
                               \end{array}
                             \right. \end{split}\eee

\begin{figure}[h] \centering
\begin{tabular}{c}
\epsfig{file=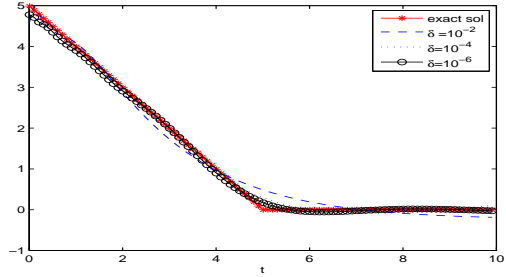,height=4cm,width=8cm,clip=} \\
\end{tabular}
\caption{Example 9: the stability of the approximate
solution}\label{fig:f9}
\end{figure}
As in Example 6 the error of the approximate solution at the
point where the function is not differentiable dominates the
error of the approximation. The reconstruction of the exact
solution can be seen in Figure 9. The detailed results are
presented in Table 9. When the double precision is used, we get
comparable results with the results shown in \cite{MCDG07}.
\begin{table}[htbp]\caption{Example 9.}\label{tab9}
\newcommand{\m}{\hphantom{$-$}}
\renewcommand{\arraystretch}{1.2} %\begin{center}{\scriptsize
\begin{tabular}{llllll}
\hline
$\dl$&$MAE$&$m_{n_\dl}$&$Iter.$&CPU time (seconds)&$a_{n_\dl}$\\

\hline \hline
$   1.00\times 10^{-2}    $&$ 2.07\times 10^{-1}    $&$ 30  $&$ 3   $&$ 6.25\times 10^{-2}    $&$ 2.00\times 10^{-6}    $\\
$   1.00\times 10^{-4}    $&$ 7.14\times 10^{-2}    $&$ 32 $&$ 4   $&$ 3.44\times 10^{-1}    $&$ 2.00\times 10^{-9}    $\\
$   1.00\times 10^{-6}    $&$ 2.56\times 10^{-2}    $&$ 54 $&$ 5   $&$ 3.75\times 10^{-1}    $&$ 2.00\times 10^{-12}    $\\
\hline \hline
\end{tabular}%}\end{center}
\end{table}

\item \textbf{\textit{ Example 10. (see \cite{BDBM79})}}\\
\bee\begin{split}f_{10}(t)&=\left\{
              \begin{array}{ll}
                t, & \hbox{$0\leq t<10$,} \\
                0, & \hbox{elsewhere,}
              \end{array}
            \right.\\
 F_{10}(p)&=\left\{
                                     \begin{array}{ll}
                                       50, & \hbox{$p=0$,} \\
                                       \frac{1-e^{-10p}}{p^2}-\frac{10e^{-10p}}{p}, & \hbox{$p>0$.}
                                     \end{array}
                                   \right. .\end{split}\eee

\begin{figure}[h!] \centering
\begin{tabular}{c}
\epsfig{file=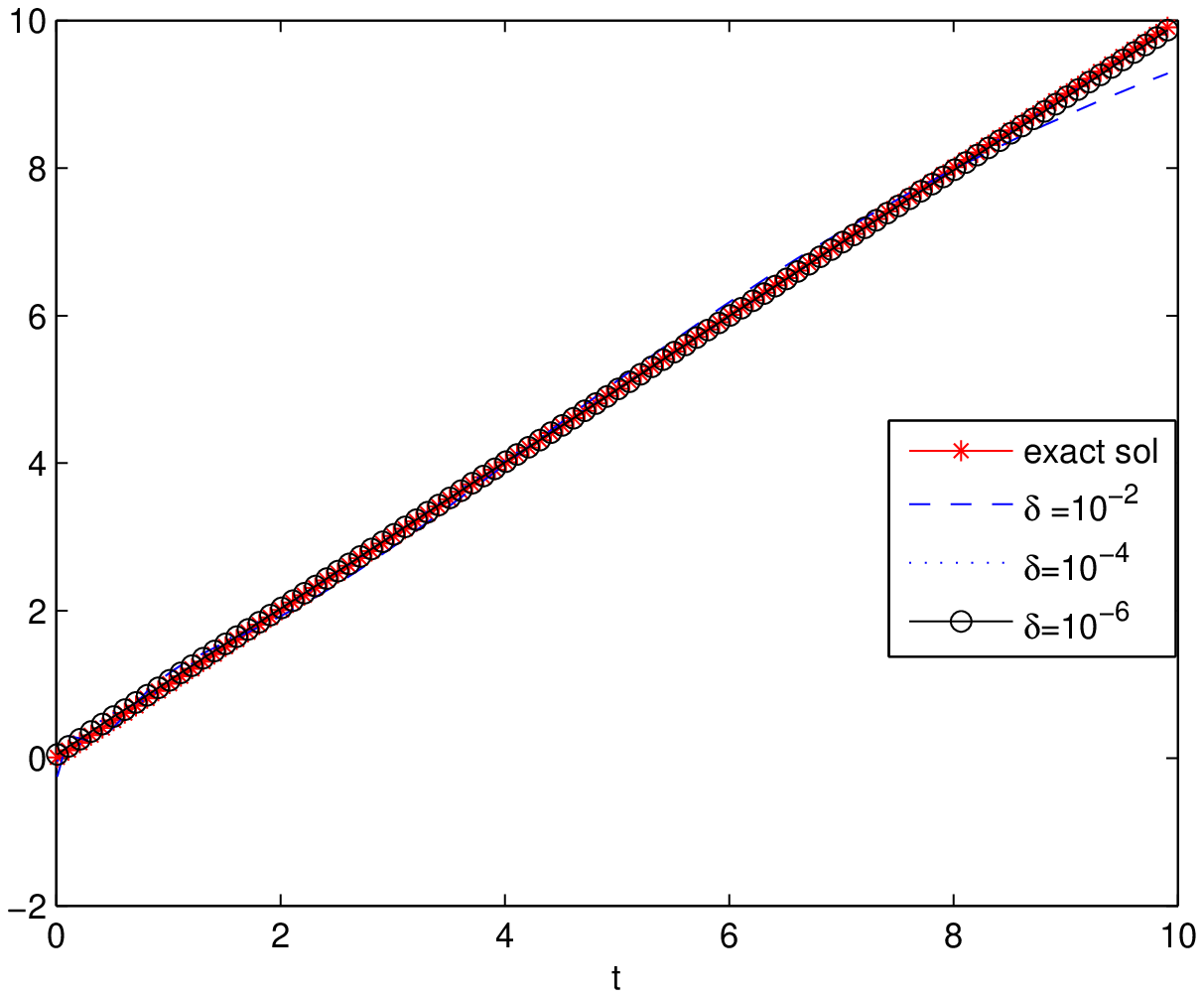,height=4cm,width=8cm,clip=} \\
\end{tabular}
\caption{Example 10: the stability of the approximate
solution}\label{fig:f10}
\end{figure}

\begin{table}[h!]\caption{Example 10.}\label{tab10}
\newcommand{\m}{\hphantom{$-$}}
\renewcommand{\arraystretch}{1.2} %\begin{center}{\scriptsize
\begin{tabular}{llllll}
\hline
$\dl$&$MAE$&$m_{n_\dl}$&$Iter.$&CPU time (seconds)&$a_{n_\dl}$\\

\hline \hline
$   1.00\times 10^{-2}    $&$ 2.09\times 10^{-1}    $&$ 30  $&$ 3   $&$ 3.13\times 10^{-2}    $&$ 2.00\times 10^{-6}    $\\
$   1.00\times 10^{-4}    $&$ 1.35\times 10^{-2}    $&$ 32 $&$ 4   $&$ 9.38\times 10^{-2}    $&$ 2.00\times 10^{-9}    $\\
$   1.00\times 10^{-6}    $&$ 3.00\times 10^{-3}    $&$ 54 $&$ 4   $&$ 2.66\times 10^{-1}    $&$ 2.00\times 10^{-9}    $\\

\hline \hline
\end{tabular}%}\end{center}
\end{table}

Table 10 shows the stability of the solution obtained by our
method with respect to the noise level $\dl$. We get an
        excellent agreement between the exact solution and the
        approximate solution for all the noise
levels $\dl$ as shown in Figure 10.

\item \textbf{\textit{ Example 11. (see \cite{BDBM79}, \cite{VVK06})}}\\
\bee\begin{split}f_{11}(t)&=\left\{
              \begin{array}{ll}
                \sin(t), & \hbox{$0\leq t<10$,} \\
                0, & \hbox{elsewhere,}
              \end{array}
            \right.\\
F_{11}(p)&=\frac{1 - e^{-10p}(p\sin(10) + \cos(10))}{1 + p^2}
.\end{split}\eee
\begin{figure}[htbp] \centering
\begin{tabular}{c}
\epsfig{file=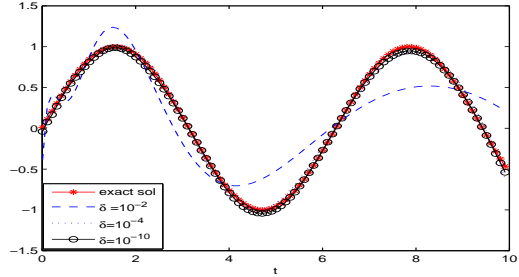,height=4cm,width=8cm,clip=} \\
\end{tabular}
\caption{Example 11: the stability of the approximate
solution}\label{fig:f11}
\end{figure}

Here the function $f_{11}(t)$ represents the class of periodic
functions. It is mentioned in \cite{VVK06} that oscillating
function can be found with acceptable accuracy only for
relatively small values of $t$. In this example the best
approximation is obtained when the noise level $\dl=10^{-6}$
which is comparable with the results given in \cite{BDBM79} and
\cite{VVK06}.  The reconstruction of the function $f_{11}(t)$
for various levels of the noise $\dl$ are given in Figure 11.
The stability of the proposed method with respect to the noise
$\dl$ is shown in Table 11. In this example the parameter
$m_{n_\dl}$ is bounded by the constant $54$ when the noise level
$\dl=10^{-6}$ and $\kappa=0.3$.

\begin{table}[h!]\caption{Example 11.}\label{tab11}
\newcommand{\m}{\hphantom{$-$}}
\renewcommand{\arraystretch}{1.2} %\begin{center}{\scriptsize
\begin{tabular}{llllll}
\hline
$\dl$&$MAE$&$m_{n_\dl}$&$Iter.$&CPU time (seconds)&$a_{n_\dl}$\\

\hline \hline
$   1.00\times 10^{-2}    $&$ 2.47\times 10^{-1}    $&$ 30  $&$ 3   $&$ 9.38\times 10^{-2}    $&$ 2.00\times 10^{-6}    $\\
$   1.00\times 10^{-4}    $&$ 4.91\times 10^{-2}    $&$ 32 $&$ 4   $&$ 2.50\times 10^{-1}    $&$ 2.00\times 10^{-9}    $\\
$   1.00\times 10^{-6}    $&$ 2.46\times 10^{-2}    $&$ 54 $&$ 5   $&$ 4.38\times 10^{-1}    $&$ 2.00\times 10^{-12}    $\\
\hline \hline
\end{tabular}%}\end{center}
\end{table}
\item \textbf{\textit{ Example 12. (see \cite{SDMR05}, \cite{BDBM79})}}\\
\bee\begin{split} f_{12}(t)&=\left\{
              \begin{array}{ll}
                t\cos(t), & \hbox{$0\leq t<10$,} \\
                0, & \hbox{elsewhere,}
              \end{array}
            \right.\\
 F_{12}(p)&= \frac{(p^2-1)-e^{-10p}(-1 + p^2 + 10p +
            10p^3) \cos(10)}{(1 + p^2)^2}\\
 &+\frac{ e^{-10p}(2p + 10 + 10p^2)\sin(10)}{(1 + p^2)^2}.
\end{split}\eee

\begin{figure}[htbp] \centering
\begin{tabular}{c}
\epsfig{file=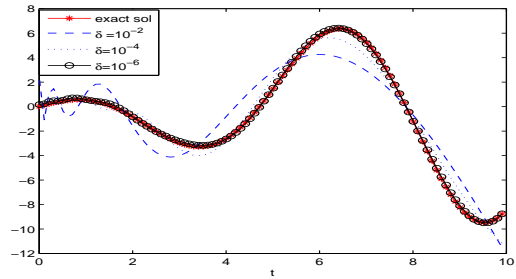,height=4cm,width=8cm,clip=} \\
\end{tabular}
\caption{Example 12: the stability of the approximate
solution}\label{fig:f12}
\end{figure}
Here we take an increasing function which oscillates as the
variable $t$ increases over the interval $[0,10)$. A poor
approximation is obtained when the noise level $\dl=10^{-2}$.
Figure 12 shows that the exact solution can be approximated very
well when the noise level $\dl=10^{-6}.$ The results of our
method are comparable with these of the methods given in
\cite{SDMR05} and \cite{BDBM79}. The stability of our method
with respect to the noise level is shown in Table 12.
\begin{table}[h!]\caption{Example 12.}\label{tab12}
\newcommand{\m}{\hphantom{$-$}}
\renewcommand{\arraystretch}{1.2} %\begin{center}{\scriptsize
\begin{tabular}{llllll}
\hline
$\dl$&$MAE$&$m_{n_\dl}$&$Iter.$&CPU time (seconds)&$a_{n_\dl}$\\

\hline \hline
$   1.00\times 10^{-2}    $&$ 1.37\times 10^{0}    $&$ 96  $&$ 3   $&$ 9.38\times 10^{-2}    $&$ 2.00\times 10^{-6}    $\\
$   1.00\times 10^{-4}    $&$ 5.98\times 10^{-1}    $&$ 100$&$ 4   $&$ 2.66\times 10^{-1}    $&$ 2.00\times 10^{-9}    $\\
$   1.00\times 10^{-6}    $&$ 2.24\times 10^{-1}    $&$ 300 $&$ 5   $&$ 3.44\times 10^{-1}    $&$ 2.00\times 10^{-12}    $\\

\hline \hline
\end{tabular}%}\end{center}
\end{table}

\item \textbf{\textit{ Example 13.}}\\
\bee f_{13}(t)=e^{-t},\quad F_{13}(p)=\frac{1}{1 + p}. \eee Here
the support of $f_{13}(t)$ is not compact. From the Laplace
transform formula one gets \bee\begin{split}
F_{13}(p)&=\int_0^\infty
e^{-t}e^{-pt}dt=\int_0^be^{-(1+p)t}dt+\int_b^\infty
e^{-(1+p)t}dt\\
&=\int_0^bf_{13}(t)e^{-pt}dt+\frac{e^{-(1+p)b}}{1+p}:=I_1+I_2,
\end{split}\eee where $\dl(b):=e^{-b}.$ Therefore, $I_2$ can be considered as noise of the data $F_{13}(p)$, i.e., \be
F_{13}^\dl(p):=F_{13}(p)-\dl(b), \ee where $\dl(b):=e^{-b}.$ In
this example the following parameters are used: $d=2$,
$\kappa=10^{-1}$ for $\dl=e^{-5}$ and $\kappa=10^{-5}$ for
$\dl=10^{-8},$ $10^{-20}$ and $10^{-30}$. Table 13 shows that
the error decreases as the parameter $b$ increases. The
approximate solution obtained by the proposed method converges
to the function $f_{13}(t)$ as $b$ increases (see Figure 13).

\begin{table}[h!]\caption{Example 13.}\label{tab13}
\newcommand{\m}{\hphantom{$-$}}
\renewcommand{\arraystretch}{1.2} %\begin{center}{\scriptsize
\begin{center}\begin{tabular}{lllll}
\hline
$b$&$MAE$&$m_\dl$& Iter& CPU time (seconds)\\
\hline \hline
$ 5 $ & $   1.487\times 10^{-2}$&   2 &4&$3.125\times 10^{-2}$\\
$ 8 $ & $   2.183\times 10^{-4} $& 2 &4&$3.125\times 10^{-2}$\\
$ 20 $ & $   4.517\times 10^{-9}$ & 2&4&$3.125\times 10^{-2}$ \\
$ 30 $ & $   1.205\times 10^{-13} $&  2&4&$3.125\times 10^{-2}$\\
\hline \hline
\end{tabular}\end{center}
\end{table}

\begin{figure}[htbp] \centering
\begin{tabular}{c}
\epsfig{file=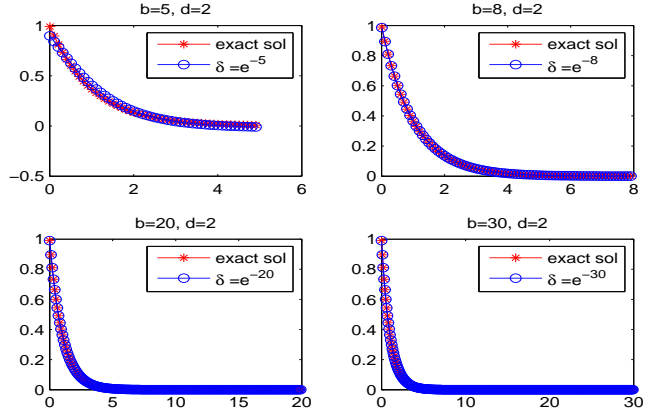,height=6cm,width=10cm,clip=} \\
\end{tabular}
\caption{Example 13: the stability of the approximate
solution}\label{fig:f13}
\end{figure}

\end{itemize}

\newpage
\section{Conclusion}
We have tested the proposed algorithm on the wide class of examples
considered in the literature. Using the rule \eqref{rulm} and the
stopping rule \eqref{dp}, the number of terms in representation
\eqref{fmdel}, the discrete data $F_\dl(p_j)$, $j=0,1,2,\hdots,m$,
and regularization parameter $a_{n_\dl}$, which are used in
computing the approximation $f_m^\dl(t)$ (see \eqref{fmdel}) of the
unknown function $f(t)$, are obtained automatically. Our numerical
experiments show that the computation time (CPU time) for
approximating the function $f(t)$ is small, namely CPU time $\leq 1$
seconds, and the proposed iterative scheme and the proposed adaptive
stopping rule yield stable solution with respect to the noise level
$\dl$. The proposed method also works for $f$ without compact
support as shown in Example 13. Moreover, in the proposed method we
only use a simple representation \eqref{fmdel} which is based on the
kernel of the Laplace transform integral, so it can be easily
implemented numerically.

\end{document}